\documentclass[11pt,a4paper,reqno]{amsart}
\usepackage[applemac]{inputenc}
\usepackage[T1]{fontenc}
\usepackage{amsmath}
\usepackage{amsthm}
\usepackage{amsfonts}
\usepackage{amssymb}
\usepackage{graphicx}
\usepackage{amsbsy}
\usepackage{mathrsfs}
\usepackage{bbm}

\newtheorem{theorem}{Theorem}[section]
\newtheorem{lemma}[theorem]{Lemma}

\newtheorem{proposition}[theorem]{Proposition}
\newtheorem{claim}{Claim}

\theoremstyle{remark}
\newtheorem{remark}[theorem]{\it \bf{Remark}\/}

\numberwithin{equation}{section}
\catcode`@=11
\def\section{\@startsection{section}{1}%
  \z@{1.5\linespacing\@plus\linespacing}{.5\linespacing}%
  {\normalfont\bfseries\large\centering}}
\catcode`@=12
\newcommand{\be}{\begin{equation}}
\newcommand{\ee}{\end{equation}}
\newcommand{\bea}{\begin{eqnarray}}
\newcommand{\eea}{\end{eqnarray}}
\newcommand{\bee}{\begin{eqnarray*}}
\newcommand{\eee}{\end{eqnarray*}}

\def\NN{\mathbb{N}}
\def\RR{\mathbb{R}}

\def\fref#1{{\rm (\ref{#1})}}

\catcode`@=11
\def\supess{\mathop{\operator@font Sup\,ess}}
\catcode`@=12

\def\NN{\mathbb{N}}

\def\RR{\mathbb{R}}

\def\e{\varepsilon}

\def\bar#1{{\overline #1}}
\def\fref#1{{\rm (\ref{#1})}}

\def\R2+{\RR ^2_+}

\def\lsl{\frac{\lambda_s}{\lambda}}

\def\lim{\mathop{\rm lim}}

\def\sup{\mathop{\rm sup}}

\def\l{\lambda}

\def\log{{\rm log}}

\def\lsl{\frac{\lambda_s}{\lambda}}
\def\xsl{\frac{x_s}{\lambda}}

\def\cal{\mathcal}

\def\matchal{\mathcal}

\def\qz{q_0}
\def\qu{q_1}
\def\pp{p}
\def\qnz{q}
\def\zz{Z}
\def\Yz{Y_0}

\title[Blow up for the critical gKdV III]{Blow up for the critical gKdV equation III: \\
exotic regimes}
\author[Y. Martel]{Yvan Martel}
\address{Universit\'e de Versailles St-Quentin and Institut Universitaire de France, LMV  CNRS UMR8100}
\email{yvan.martel@uvsq.fr}
\author[F. Merle]{Frank Merle}
\address{Universit\'e de Cergy Pontoise and Institut des Hautes Etudes Scientifiques, AGM CNRS UMR8088}
\email{merle@math.u-cergy.fr}
\author[P. Rapha\"el]{Pierre Rapha\"el}
\address{Universit\'e Paul Sabatier and  Institut Universitaire de France, Institut de Math\'ematiques de Toulouse, CNRS UMR          5219}
\email{pierre.raphael@math.univ-toulouse.fr}
\begin{document}

\begin{abstract}
We consider the blow up problem in $H^1$ for the $L^2$ critical (gKdV) equation in the continuation of \cite{MMR1}, \cite{MMR2}.
We know from \cite{MMR1} that the unique and stable blow up rate for $H^1$ solutions close to the solitons with strong decay on the right is $$\|u_x(t)\|_{L^2} \sim\frac{1}{T-t}   \quad \hbox{as $t\uparrow T<+\infty$}.$$ 
In this paper,  we construct non-generic blow up regimes in $H^1$  by considering initial data with explicit slow  decay on the right in space. We obtain finite time blow up solutions with speed
$$
\|u_x(t)\|_{L^2} \sim \frac{1}{(T-t)^\nu} \quad \hbox{as $t\uparrow T<+\infty$}, \ \ \nu>\frac{11}{13},
$$
as well as global in time growing up solutions with both exponential growth $$
\|u_x(t)\|_{L^2} \sim e^{t} \quad \hbox{as $t\to +\infty$},
$$
or any power growth
$$
\|u_x(t)\|_{L^2} \sim t^{\nu} \quad \hbox{as $t\to +\infty$}, \ \ \nu>0.
$$
These solutions can be taken with initial data arbitrarily close in $H^1$ to the ground state solitary wave.
\end{abstract}

\maketitle
\section{Introduction}
\subsection{Setting of the problem}
We consider the $L^2$-critical generalized Korteweg--de Vries equation (gKdV) 
\begin{equation}\label{kdv}
{\rm (gKdV)}\quad  \left\{\begin{array}{ll}
 u_t + (u_{xx} + u^5)_x =0, \quad & (t,x)\in [0,T)\times {\mathbb R}, \\
 u(0,x)= u_0(x), & x\in {\mathbb R}.
\end{array}
\right.
\end{equation}
The Cauchy problem is locally well-posed in the energy space $H^1$ from Kenig, Ponce and Vega \cite{KPV,KPV2}. Given $u_0 \in H^1$, there exists a unique\footnote{in a certain sense} maximal solution $u(t)$ of \eqref{kdv} in $C([0,T), H^1)$ with either $T=+\infty$, or $T<+\infty$ and then $\lim_{t\to T} \|u_x(t)\|_{L^2} = +\infty$.\\
For $H^1$ solution, the mass and the energy are conserved by the flow: $\forall t\in [0,T)$,  $$M(u(t))=\int u^2(t)= M(u_0), \ \ E(u(t))= \frac 12 \int u_x^2(t) - \frac 16 \int u^6(t)= E(u_0).$$
Equation \eqref{kdv} has the following invariances: if $u(t,x)$ is solution of \eqref{kdv} then $-u(t,x)$, $ u(-t,-x)$ and 
 $$\lambda_0^{\frac12}u(\lambda_0^3 (t-t_0),\lambda_0 (x -x_0)), \ \ (\l_0,t_0,x_0)\in \RR^*_+\times \RR\times \RR$$ are also  solutions of \eqref{kdv}.\\
The family of  traveling wave solutions of \eqref{kdv},  called solitons, plays a distinguished role in the analysis:
$$u(t,x) = Q_{\l_0}(x-\l_0^{-2} t-x_0), \ \ (\l_0,x_0)\in\RR^*_+\times \RR,$$ with  
\begin{equation}\label{eq:Q}
Q_\l(x) = \frac{1}{\l^{\frac 12} }Q \left(\frac x\l\right) ,\quad Q(x) =   \left(\frac {3}{\cosh^{2}\left( 2 x\right)}\right)^{\frac14}, \quad
Q''+Q^5 =Q.
\end{equation}
It is well-known that the function $Q$ is related to the following sharp Gagliardo-Nirenberg inequality (\cite{W1983})
\be\label{gn}
\forall v\in H^1,\quad
\int |v|^6 \leq \int v_x^2 \left(\frac {\int v^2}{\int Q^2}\right)^2.
\ee
Moreover,  from   \eqref{gn},   mass and energy conservations, for
initial data in $H^1$  such that $\|u_0\|_{L^2}<\|Q\|_{L^2}$, the corresponding solution $u(t)$ of \eqref{kdv} is bounded in $H^1$ and thus globally defined in time.

\subsection{On the classification of the flow near $Q$}  
For 
\begin{equation}
\label{utwosmall}
\|Q\|_{L^2}<\|u_0\|_{L^2}<\|Q\|_{L^2}+\alpha_0, \ \ \alpha_0\ll1 
\end{equation}
the blow  up problem 
has been first studied in a series of works by Martel and Merle \cite{MMjmpa,MMgafa,Mjams,MMannals, MMjams}.
In particular, from a  rigidity theorem around solitons (\cite{MMjmpa}),  the first  proof of blow up in finite or infinite time was obtained (\cite{Mjams}) for 
initial data 
\be\hbox {$u_0 \in H^1$ such that \eqref{utwosmall} and  $E(u_0)<0$.}
\label{classe}
\ee 
Recently, in \cite{MMR1,MMR2},  the authors of the present paper have revisited the blow up analysis for data near the ground state. First, in the so-called minimal mass case $\|u_0\|_{L^2}= \|Q\|_{L^2}$, the following existence and uniqueness results  complement   results in \cite{MMduke}.

\medskip

\noindent{\bf Minimal mass blow up solution  {\rm (\cite{MMR2}, \cite{MMduke})}.} 
{\it {\rm (i) Existence.} There exists a solution $S(t)\in \mathcal C((0,+\infty),H^1)$  to \fref{kdv} with minimal mass $\|S(t)\|_{L^2}=\|Q\|_{L^2}$ such that  
 \be\label{th1:1.4}\|  S_x(t)\|_{L^2}\sim \frac{\|Q'\|_{L^2}}{ t}\ \ \mbox{as} \  t\downarrow 0.
 \ee  \\
{\rm  (ii) Uniqueness.}
Let $u$ be an $H^1$  blow up solution to \eqref{kdv}  with minimal mass $\|u(t)\|_{L^2} = \|Q\|_{L^2}$. Then $u=S$ up to the invariances of the (gKdV) equation.}

\medskip

Second,    \cite{MMR1,MMR2} yield a classification of the flow for initial data close to $Q$ {\it with decay on the right  in space}. More precisely, let
$$
\mathcal{A}=
\left\{
u_0=Q+\e_0   \hbox{ with } \|\e_0\|_{H^1}<\alpha_0 \hbox{ and }
\int_{y>0} y^{10}\e_0^2< 1
\right\},
$$
$$\mathcal T_{\alpha^*}=\left\{u\in H^1\ \ \mbox{with}\ \ \inf_{\l_0>0, \, x_0\in \RR}
\left\|u- Q_{\l_0}\left( .-x_0 \right) \right\|_{L^2} <\alpha^*\right\}.$$
Then the following classification  result holds:

\medskip

\noindent{\bf Classification in $\mathcal A$ {\rm (\cite{MMR1,MMR2})}.} 
Let $0<\alpha_0\ll\alpha^*\ll 1$.
{\sl Let   $u_0 \in \mathcal{A}$ and 
  $u \in \mathcal C ([0,T),H^1)$ be the corresponding solution of \eqref{kdv}.
Then, one of the following three scenarios occurs:

\medskip

\noindent{\bf (Blow up)}  For all $t\in [0,T),$ $u(t)\in \mathcal{T}_{\alpha^*}$ and the solution blows up in finite time $T<+\infty$ with
\be\label{papI1}
\|u_x(t)\|_{L^2} \sim \frac {\|Q'\|_{L^2}}{\ell_0 (T-t)} \quad \hbox{as $t\uparrow T$ for some $\ell_0>0$.} 
\ee

\noindent{\bf (Soliton)}  The solution is global and
\begin{equation}
  u(t,  \cdot +x(t))\to Q_{\l_\infty} \hbox{ in $H^1_{\rm loc}$ as $t\to +\infty$ for  $|\l_\infty-1| + |x'(t)-1|\lesssim \delta(\alpha_0)$},
\end{equation} 
where $\delta(\alpha_0)\to 0$  as $\alpha_0\to 0$.
\smallskip

\noindent{\bf (Exit and $S$ dynamics)} The solution  $u$ exits the tube   $\mathcal{T}_{\alpha^*}$ at some time  $t_u\in (0,T)$, and  there exist    $\l_u>0$, $x_u\in \RR$,  such that
$$\| {\l_u^{\frac 12}} u(t_u, \l_u x+x_u) - S(t^*,x)\|_{L^2} \leq \delta(\alpha_0),$$
where $\delta(\alpha_0)\to 0$  as $\alpha_0\to 0$ and where $t^*>0$ depends only on $\alpha^*$.\\
Moreover, assume that $S$ scatters at $+\infty$, then  $u$ is global  and scatters at $+\infty$.
  }
  
  \medskip
  
In particular, this   indicates that for initial data in $\mathcal{A}$, only one type of blow up is possible. In this paper, we prove that for initial data in $H^1$, but   with slow decay, different blow up behaviors are possible close to solitons.
It means that the decay assumption in the definition of $\mathcal A$ is not a technical one.

\subsection{Exotic blow up regimes}
We now consider initial data $u_0\not \in \mathcal{A}$ in the sense that they display an explicit slow decay on the right. Our main result in this paper says that the  blow up rate $\frac 1{(T-t)}$, which is universal in $\mathcal A$, is not valid anymore for such initial data. Indeed, we  produce a wide range of different blow up rates, including grow up in infinite time.

\begin{theorem}[Exotic blow up regimes]\label{th:1}\ \\
{\rm (i)  Blow up in finite time:} for any $\nu > \frac {11}{13}$,  there exists $u \in  C((0,T_0],H^1)$   solution  of \eqref{kdv} blowing up at $t=0$ with
\be\label{th:1:1}
\|u_x(t)\|_{L^2} \sim t^{-\nu}\ \ \mbox{as}\ \ t\downarrow 0^+.\ee
{\rm (ii)  Grow up in infinite time:} there exists  $u\in  C([T_0,+\infty),H^1)$ solution of \eqref{kdv} growing  up at $ +\infty$ with
\be\label{th:1:4}
\|u_x(t)\|_{L^2} \sim e^{t} \ \ \mbox{as}\ \ t\to +\infty.
\ee
For any $\nu>0$,  
there exists  $u\in  C([0,+\infty),H^1)$ solution of \eqref{kdv} growing  up at $ +\infty$ with
\be\label{th:1:5}
\|u_x(t)\|_{L^2} \sim t^{\nu}\ \ \mbox{as}\ \ t\to +\infty.
\ee
Moreover, such solutions can be taken arbitrarily close in $H^1$ to the family of solitons.
\end{theorem}

\noindent {\bf Comments on Theorem \ref{th:1}.}
\smallskip

{\it 1. Sharpness of the results in \cite{MMR1, MMR2}}. Theorem \ref{th:1} above shows the optimality of the results in  \cite{MMR1, MMR2} since it proves that some decay assumption (such as $u_0\in \mathcal A$) is required to obtain a unique  stable blow up rate $1/(T-t)$. This is in contrast with the nonlinear Schr\"odinger equation, for which the stable blow up rate is obtained in $H^1$, without additional decay assumption (see \cite{MRjams} and references therein). Note from the proof that the solutions obtained in Theorem \ref{th:1} are expected to be unstable (except may be for $\nu<1$ in \eqref{th:1:1}). Indeed, they are constructed using a topogical argument involving two possible directions of instability.\\

{\it 2}.
It is proved in \cite{Mjams, MMannals} that initial data $u_0$ such that \eqref{classe} generate solutions that blow up in finite or infinite time. The proof is by obstruction and Liouville classification and does not provide any estimate on the blow up speed. This $H^1$ result is also sharp in the sense that  from Theorem~\ref{th:1},  both finite or infinite time blow up may occur in $H^1$.  All these results  thus complement each other.\\

{\it 3. On the role of tails.} As one can see from the proof of Theorem \ref{th:1}, the blow up rate is directly related to the precise behavior of the initial data on the right. In particular, other type of blow up speeds can be produced by similar arguments by adjusting the tail of the initial data. A   similar phenomenon was observed for global in time growing up solutions to the parabolic energy critical harmonic heat flow by Gustafson, Nakanishi and Tsai \cite{NT2}. There an explicit formula on the growth of the solution at infinity is given directly in terms of the initial data which is conceptually very similar to what we observe for (gKdV). 

Recall that continua of blow up rates were   observed in   pioneering works by Krieger, Schlag and Tataru \cite{KST}, \cite{KST2} for   energy critical wave problems (see also Donninger and Krieger \cite{DK}). We also refer to Fila et al. \cite{FKWY} for a   formal approach in the case of the energy critical heat equation. All these results point out that the sole critical topology is not enough to classify the flow near the ground state.\\

{\it 4. On the   decay assumption}. In \cite{MMR1} (see the definition of $\mathcal A$), the assumption $\int y^{10} \e^2<1$ is not sharp.
In Theorem \ref{th:1}, the solution contains a tail  of the form $x^{-\theta}$ for $x\gg 1$, where $\theta\in(1,\frac {29}{18})$. By now, it is not clear what is the sharp decay assumption on the initial data required to get  the stable blow up rate in \cite{MMR1}.\\

\noindent{\bf Aknowledgments}. P.R. is supported by the French ERC/ANR project SWAP. Part of this work was completed was P.R. was visiting the Mathematics Department at MIT which he would like to thank for its kind hospitality. This work is also supported by the project ERC 291214 BLOWDISOL.\\

\noindent{\bf Notation}. 
For $f,g\in L^2$, we note the scalar product: $$(f,g)=\int f(x)g(x)dx.$$ We introduce the generator of the $L^2$ scaling symmetry $$ \Lambda f=\frac12f+yf'. $$ We let the linearized operator close to the ground state be:
\be\label{deflplus}  
Lf=-f''+f-5Q^4f.
\ee  
 For a given  small constant $0<\alpha^*\ll1 $,  $\delta(\alpha^*)$ denotes a   small constant with $$\delta(\alpha^*)\to 0\ \ \mbox{as}\ \ \alpha^*\to 0.$$ We denote by  ${\mathbf{1}}_I$   the characteristic function of the interval $I$.


\subsection{Strategy of the proof}\label{strategy} 
\label{vneovnoen}

\noindent(i) \emph{Definition and role of the slow decaying tail}.
 Given $c_0\in \RR$,  $x_0\gg 1$, $\theta>1$, we fix
 a smooth function $f_0$ which corresponds to a slowly decaying tail:
 \be\label{deffzintro}
f_0(x)= c_0 x^{-\theta} \hbox{ for $x>\frac {x_0}2$},\quad
f_0(x)=0 \hbox{ for $x<\frac {x_0}4$}, 
\ee
and  
 $\qz$   the solution of 
\be\label{defqzintro}
\partial_t \qz + \partial_x(\partial_x^2 \qz + \qz^5)=0,\quad
\qz(0,x)= f_0(x).
\ee
We then consider the solution to \fref{kdv} with initial data $Q+f_0$ and claim that it admits a decomposition of the form  
\be\label{uu}
u(t,x)  =  \frac 1{\l^{\frac 12}(t)} \left( Q_{b(t)} + \l^{\frac 12}(t) q_0(t,x(t)) \Yz + \e   \right) \left( t,\frac {x-x(t)}{\l(t)}\right)
+q_0(t,x)
\ee
for some $$\|\e(t)\|_{H^1}\ll 1$$
and where $Y_0$ is a fixed function (see Lemma \ref{lemmainvert} for the definition of $Y_0$ and Proposition \ref{le:2} for the justification of this correction term). 
An essential feature of the nonlinear (gKdV) flow is that $q_0(t,x)$ conserves for $x\gtrsim t$  the slow decay of   $f_0(x)$ (see Lemma \ref{le:qz}). This tail then acts like an external force on the coupled  system of modulation equations driving $(b(t),\l(t),x(t))$ and modify its behavior.

\medskip

\noindent (ii) \emph{Dynamical system perturbed by  a  tail on the right.} Let us consider the global renormalized time 
\be
\label{sharpintro}
 \frac {ds}{dt} = \frac 1 {\l^3}. \ee
Then, explicit computations similar to the ones in \cite{MMR1} yield to leading order (neglecting $\e$ and higher order terms in $(b,\l,x)$) the set of coupled modulation equations in the setting of the decomposition \fref{uu}:
  \be\label{dynsyst}
 \quad \lsl+b=0,\quad x_s = \l,\quad \frac d{ds} \left( \frac b{\l^2} + \frac 4 {\int Q} c_0   \l^{-\frac 32} x^{-\theta} \right)=0.
 \ee
 This system is to be compared to the unperturbed one obtained in \cite{MMR1}, for $u_0\in \mathcal{A}$ (without tail):
 \be\label{dynsystbis}
 \frac {ds}{dt} = \frac 1 {\l^3}, \quad 
 \lsl+b=0,\quad x_s = \l,\quad \frac d{ds}   \left(\frac b{\l^2}\right)  =0,\ee
which leads to the universal blow up regime $$\frac b{\l^2} = \ell_0, \ \ \l(t) = \ell_0 (T-t)\ \ \mbox{for some}\ \ \ell_0>0.$$ We now integrate explicitly \fref{dynsyst} and fit the parameters of the tail $(c_0,\theta)$ to obtain the blow up regimes described in Theorem \ref{th:1}. Integrating  in $s$, we find
$$
 \frac b{\l^2} + \frac 4 {\int Q} c_0   \l^{-\frac 32} x^{-\theta} =\ell_0,
$$
where $\ell_0$ is a constant. We focus on the threshold regime $\ell_0=0$ leading to:
 \be
 \lsl+b=0,\quad x_s = \l,\quad 
 \frac b{\l^2} + \frac 4 {\int Q} c_0   \l^{-\frac 32} x^{-\theta}=0,
 \ee
 which can now be integrated as follows:
 $$
 \l^{-\frac 12}\l_s = \frac 4 {\int Q} c_0   \l x^{-\theta}
 = \frac 4 {\int Q} c_0 x_s x^{-\theta}
 $$
 or equivalently after integration:
 $$
   \l^{\frac 12}(s) + \frac 2 {\int Q}   \frac {1}{\theta-1}c_0 x^{-\theta+1}(s) = \ell_1.
 $$
 We focus again on the threshold regime $\ell_1=0$ leading to:
 $$
   \l^{\frac 12}(s) = -\frac 2 {\int Q}   \frac {1}{\theta-1}c_0 x^{-\theta+1}(s).
 $$ 
We see that $c_0<0$ is necessary at this point and
$$
x_s(s) = \l(s) = \left(\frac 2 {\int Q}   \frac {1}{\theta-1}c_0\right)^2 x^{-2 \theta+2}(s).
$$
By integration on $[s_0,s]$, choosing $ x^{2\theta-1}(s_0)=(2\theta- 1)\left(\frac 2 {\int Q}   \frac {1}{\theta-1}c_0\right)^2 s_0$, we obtain 
$$
x^{2\theta-1}(s) =  (2\theta- 1)\left(\frac 2 {\int Q}   \frac {1}{\theta-1}c_0\right)^2 s,
$$
and thus
$$
\l(s) =  (2\theta- 1)^{\frac {-2\theta+2}{2\theta-1}} \left(\frac 2 {\int Q}   \frac {1}{\theta-1} c_0\right)^{\frac {2}{2 \theta-1}}  s^{- \frac {2(\theta-1)}{2\theta-1}}.
$$
Set
$$
\beta =   \frac { 2(\theta-1)} {2\theta-1},\quad
\theta = \frac {1-\frac \beta2}{1-\beta}, \quad c_0 =-\frac { \int Q}2(\theta-1) (2\theta-1)^{\theta-1} ,$$
so that 
\be\label{eq:blx}
\l(s) = s^{-\beta}, \quad 
x(s) = \frac 1{1-\beta} s^{1-\beta},\quad
b(s) =\frac {\beta}s.
\ee
Of course, one can check directly that \eqref{eq:blx} are solutions of the system \eqref{dynsyst}
but the above computation reveals the two instability directions
\be
\label{unstabomodesintro}
\ell_0=\ell_1= 0,
\ee
and justifies the use of a topological argument to construct the solution. 
 
 \medskip
 
 \noindent (iii) \emph{Control of the remainder term.}  We now aim at constructing an exact solution which  corresponds to   control  the remainder term $\e(t,x)$. Note that we may now choose $\e_0(x)$ to be {\it well localized on the right}, and we therefore adapt the machinery developed in \cite{MMR1} to construct a mixed energy/Virial functional $$\matchal F\sim \int\psi\e_y^2+\varphi\e^2-\frac 13  \psi\left[(Q_b+\e)^6-Q_b^6-6Q_b^5\e\right]$$ for well chosen cut off functions $(\psi,\varphi)$ which are exponentially decaying to the left, and polynomially growing to the right. Roughly speaking, in the above regime \eqref{eq:blx}, this functional enjoys two fundamental properties:
\\
-- Coercivity :
 $$\matchal F \gtrsim \|\e\|_{H^1_{\rm loc}}^2.$$
 -- Lyapounov monotonicity :  
\be
\label{cnkneneoe}
\frac{d}{ds}\left\{s^j {\mathcal F}\right\}+ s^j \|\e\|_{H^1_{\rm loc}}^2
\lesssim s^{j-4}, \ \ j\geq 0.
\ee
Time integration of the monotonicity formula \fref{cnkneneoe} in the regime dictated by \fref{eq:blx} yields  sufficient uniform estimates on $\e$. Therefore, it only remains to adjust the initial parameters $(b(s_0),\l(s_0))$ in order to asymptotically satisfy the unstable conditions \fref{unstabomodesintro}. This is achieved using a simple topological argument, as in \cite{MMC} but in a blow up setting (see also \cite{RH}, \cite{S}, \cite{MRR}, \cite{CoteZaag}).
 
\medskip

\noindent (iv) \emph{Conclusion of the proof returning to the original time variable.} The above strategy is implemented for all $0<\beta<\frac{11}{20}$. Now we show how the behavior of the parameters  \fref{eq:blx} (see the precise estimates in \eqref{sharp}) in renormalized time leads to the  scenarios of Theorem \ref{th:1} in the original time $t$  (after possible scaling and time translation to adjust constants).

\smallskip

\noindent{-- Blow up in finite time:}
for 
$
\frac 13 < \beta < \frac {11}{20}.
$
From \eqref{sharpintro}, \fref{eq:blx}: $$\int_{s_0}^{+\infty}    {\l^3(s)} ds=T<+\infty$$ and the 
solution $u(t)$ blows up in finite time $T$.
Moreover,
$$  T-t=\int_{s(t)}^{+\infty}{\l^3(s')}{ds'} 
\sim \frac {s^{ -( 3 \beta-1)}}{3\beta-1}  , \quad \l(t) \sim  \left[ (3\beta-1)(T-t)\right]^{\frac {\beta}{3\beta-1}}, 
$$
which implies  $\|u_x(t)\|_{L^2} \sim \|Q'\|_{L^2} \l^{-1} (t) \sim C (T-t)^{-\nu}$ for any $\nu\in (\frac {11}{13},+\infty).$ 

\medskip

\noindent{-- Grow up in infinite time:}  for $ 
\beta = \frac 13,
$ 
the solution $u(t)$ is global in time since $\int_{s_0}^{+\infty}  {\l^3(s)} ds =+\infty$.
Moreover, for some $c_0$ and some $c_1 >0$,
$$
t=\int_{s_0}^{s(t)}{\l^3(s')}{ds'} 
=  \log s + c_0 +O(s^{-\frac 1{10}}),\quad  s\sim  c_1 e^t  ,\quad
 \l(t)\sim  c_1^{-\frac 13} e^{-\frac t 3}  .
$$
This means grow up in infinite time for $u(t)$ with exponential growth. Scaling and time translation   lead to any exponential rate $e^{-ct}$, $c>0$. Finally, for $0<\beta< \frac 13,$
we also obtain a global solution $u(t)$  since
$$\int_{s_0}^{+\infty} \l^3(s) ds \geq 2^{-3} \int_{s_0}^{+\infty} s^{-3 \beta} ds = +\infty,$$
and 
$$t= \int_{s_0}^{t(s)} \l^3(s') ds' \sim 
\frac 1{1-3 \beta} s^{1-3 \beta}  ,
\quad  \l(t) \sim  (1-3 \beta)^{\frac \beta{1-3 \beta}} t^{\frac \beta{1-3 \beta}}  ,
$$
which means   grow up rates $t^\nu$ at $+\infty$, for any $\nu=\frac {\beta}{1-3\beta}>0$.


\section{Decomposition of the solution}


This section is devoted to the study of the geometrical decomposition \fref{uu}, and in particular the derivation of the modulation equations.


\subsection{Inversion of $L$ and $Q_b$ profiles}
 

Let the functional space
       ${\mathcal Y}$ be the set of functions $f\in \mathcal C^\infty({{\mathbb R}},{{\mathbb R}})$ such that
\begin{equation}\label{Y}
    \forall k\in \mathbb{N},\ \exists C_k,\, r_k>0,\ \forall y\in {{\mathbb R}},\quad |f^{(k)}(y)|\leq C_k (1+|y|)^{r_k} {e^{-|y|}},
\end{equation}
and $L$ be the linearized operator close to $Q$ given by \fref{deflplus}. We claim:

\begin{lemma}[Invertibility of $L$]
\label{lemmainvert}
{\rm (i)} There exists a unique $\Yz\in \mathcal Y$, even, such that
\be \label{YQ}
L \Yz= 5  Q^4 ,\quad
(Q,\Yz)=-\frac 34 \int Q  .
\ee
{\rm (ii)} There exists a unique function $P$ such that $P' \in \mathcal{Y}$ and
\be
\label{eq:23}
 (L  P)'={\Lambda} Q, \ \  \lim_{y \to -\infty}   P(y) = \frac 12 \int Q,  \ \ \lim_{y \to +\infty}   P(y) =0,
 \ee
 \be
 \label{PQ}
 (P, Q) = \frac 1{16} \left(\int Q\right)^2 > 0,\quad (P,  Q')=0.
 \ee
\end{lemma}
\begin{proof}
Note that the existence and uniqueness of $\Yz$ follows readily from standard properties of the operator $L$ (see e.g. \cite{MMR1}).
Moreover,
$$
(Q,\Yz)= -\frac 12 (L\Lambda Q,\Yz) =-\frac 12(\Lambda Q,5Q^4 )
 =-\frac 34 \int Q^5=-\frac 34 \int Q.
$$
Part (ii)   is taken from \cite{MMR1}, Proposition 2.2.
\end{proof}

A simple consequence of Lemma \ref{lemmainvert} (ii) is the existence of a one parameter family of approximate self similar profiles $b\mapsto Q_b$, $|b|\ll 1$, which provide the leading order deformation of the ground state profile $Q=Q_{b=0}$ in the blow up regimes.  More precisely, let $\chi\in {\cal C}^{\infty}({\mathbb R})$ be such that $0\leq \chi \leq 1$, $\chi'\geq 0$ on $\mathbb{R}$,
$\chi\equiv 1$ on $[-1,+\infty)$, $\chi\equiv 0$ on $(-\infty,-2]$, and define
\begin{equation}\label{defgamma} 
\chi_b(y)= \chi\left(|b|^{\gamma} {y} \right),\ \ \gamma = \frac 34.
\end{equation}
The following Lemma is proved in \cite{MMR1}:

\begin{lemma}[Approximate self-similar profiles $Q_b$, \cite{MMR1}]
\label{cl:2}
Let 
\begin{equation}\label{eq:210}
\quad 
Q_b(y) = Q(y) + b \chi_b(y)   P (y).
\end{equation}
Then:\\
{\rm (i) Estimates on $Q_b$:} For all $y \in \mathbb{R}$,
\begin{align}
&	|Q_b(y)|\lesssim  e^{-|y|} + |b| \left(  {\mathbf{1}}_{[-2,0]}(|b|^\gamma y) +  e^{-\frac {|y|}{2}} \right),
\label{eq:001}\\
&	|Q_b^{(k)}(y)|\lesssim  e^{-|y|} +   |b|  e^{-\frac {|y|}{2}} +|b|^{1+k \gamma} {\mathbf{1}}_{[-2,-1]}(|b|^\gamma y),\quad  \text{for $k\geq 1$}.
\label{eq:002}
\end{align}
\noindent
{\rm (ii) Equation of $Q_b$:} let
\begin{equation}\label{eq:201}
-\Psi_b=\left(Q_b''- Q_b+ Q_b^5\right)'+b {\Lambda} Q_b,
\end{equation}
then, for all $y\in \mathbb{R}$,
\begin{align}
& |\Psi_b(y)|\lesssim |b|^{1+ \gamma}  {\mathbf{1}}_{[-2 ,-1]}(|b|^{\gamma}y) 
+  b^2  \left(e^{-\frac {|y|} 2}+ {\mathbf{1}}_{[- 2,0]}(|b|^{\gamma}y) \right),\label{eq:202}
\\
& |\Psi_b^{(k)}(y)|\lesssim    |b|^{1+ (k+1) \gamma}  {\mathbf{1}}_{[-2 ,-1]}(|b|^{\gamma}y)  + b^2   e^{-\frac {|y|} 2} ,\quad 
 \text{for $k\geq 1$}. \label{eq:203}
\end{align}
{\rm (iii) Mass and energy properties of $Q_b$:}
\begin{align}
& \left|\int Q_b^2 - \left( \int Q^2 + 2b \int PQ \right)\right| \lesssim |b|^{2-\gamma}, \label{eq:204}\\
& \left| E(Q_b) + b \int PQ\right|\lesssim b^2 .
\label{eq:205}
\end{align}
\end{lemma}

\subsection{Definition of the tail on the right} We now introduce the slowly decaying tail on the right. Let $c_0<0$,  $x_0\gg 1$, $\theta>1$  and let
$f_0$ be a smooth function such that 
 \be\label{deffz}
f_0(x)=\left\{\begin{array}{ll}c_0 x^{-\theta} & \hbox{ for $x>\frac {x_0}2$},\\ 0 & \hbox{ for $x<\frac {x_0}4$,}\end{array}\right.
\ee
and 
\be\label{estfz}
\left|\frac {d^{k}f_0}{dx^k} (x)\right| \lesssim   |x|^{-\theta - k}, \ \ \forall (x,k)\in \RR\times \NN.
\ee
Let $\qz$ be the solution of 
\be\label{defqz}
\partial_t \qz + \partial_x(\partial_x^2 \qz + \qz^5)=0,\quad
\qz(0,x)= f_0(x).
\ee
A simple consequence of local energy estimates for (gKdV) is the propagation of the tail on the right:

\begin{lemma}[Asymptotic behavior of $\qz$]\label{le:qz}
The solution $\qz$ of \eqref{defqz} is global, smooth and bounded  in $H^1$.
Moreover, $\forall t\geq 0,\  \forall x>\frac t 2+\frac {x_0}2$,
\be\label{leqz3}
\forall k\geq 0,\quad
|\partial_x^k \qz(t,x) -f_0^{(k)}(x) | \lesssim   t^{\frac 38} x^{-\theta-\frac {19}8 - k }
\lesssim x^{-\theta-2 - k }, 
\ee
\be\label{leqz4}
|\partial_t\qz(t,x) | \lesssim  x^{-\theta-3}
.
\ee
\end{lemma}
See proof of Lemma \ref{le:qz} in Appendix A.

\subsection{Decomposition of the solution}

Let $c_0\in \RR$, $\l_0\ll 1$ and $x_0\gg 1$. Consider $u(t,x)$ a solution of \eqref{kdv} and set
\be\label{defw}
w(t,x)= u(t,x) - \qz(t,x).
\ee
We assume that $w$ is close to $Q$ in the following sense : there exist $(\lambda_1(t),x_1(t))\in \RR^*_+\times \RR$ and $\e_1(t)$ such that
\begin{align}\label{decompo0}
\forall t\in [0,t_0], \ \ &
\l_1(t)<\frac {10}9 \l_0,\quad x_1(t)> \frac 9{10} x_0,
\\  \label{decompo1}
& w(t,x)=\frac{1}{\lambda^{\frac 12}_1(t)}(Q+\e_1)\left(t,\frac{x-x_1(t)}{\lambda_1(t)}\right)
\end{align}
with 
\be
\label{hypeprochien}
\forall t\in [0,t_0],\ \
 \|\e_1(t)\|_{L^2}+\left(\int (\partial_y \e_1)^2 e^{-\frac {|y|}{2}} dy \right)^{\frac 12}\leq \alpha^*
\ee
for some small enough universal constant $\alpha^*>0$. We collect in the following Proposition the standard preliminary estimates on this decomposition, and derive in particular the set of modulation equations as a consequence of a suitable choice of orthogonality conditions for the remainder term.

  \begin{proposition}[Preliminary estimates and modulation equations]
\label{le:2}
Assume \fref{decompo0}--\fref{hypeprochien} for $\alpha^*$ small enough, and assume $x_0$ large enough and $\l_0$ small enough.

\noindent{\em (i) Decomposition:} There exist $\mathcal C^1$ functions $({\lambda}, x,{b}):[0,t_0]\to (0,+\infty)\times \mathbb{R}^2$ such that
\begin{equation}\label{defofeps}
\forall t\in [0,t_0],\quad
{\lambda}^{\frac 12}(t) w(t,{\lambda}(t) y + x(t)) =  Q_{{b}(t)}(y)
+\pp(t) \Yz(y) +\e(t,y),
\end{equation}
where $Y_0$ is given by \fref{YQ}, 
\be\label{defp}
\pp(t)= \qnz(t,0),\quad 
\qnz(t,y) = \l^{\frac 12}(t) \qz(t,\l(t) y +x(t)),
\ee
and   $\e(t,y)$
satisfies  
\be
\label{ortho1}
(\e(t), y {\Lambda}   Q )=(\e(t),\Lambda Q) =(\e(t),Q)=0,
\ee
\be\label{lambdax}
\l(t)<\frac 54 \l_0,\quad x(t)>\frac 45 x_0.
\ee
{\em (ii)   Estimates induced by the conservation laws:}
\be
\label{twobound}
  \|\e(s)\|^2_{L^2}\lesssim \left|\int u_0^2-\int Q^2\right| + |b(s)| +|p(s)| + x_0^{-\theta+\frac 12},
\ee
\be
\label{energbound}
  \frac 1{\l^2}\| \e_y(s)\|_{L^2}^2       \lesssim |E(u_0)| + \left|  \frac b{\l^2}  {+} c_0 \frac 4{\int Q}\l^{-\frac 32} x^{-\theta} \right|  
+ \frac {b^2}{\l^2}+  
 \frac { |p| }{x^2} + \frac { |p| }{\l x}  + \frac {p^2}{\l^2}  +  x_0^{-\theta-\frac 12} 
 .
   \ee
 
\noindent {\em (iii) Modulation equations:} 
Assume
\be\label{surX}
\forall t\in [0,t_0],\quad x(t) > \frac 23 t + \frac 23 {x_0}.
\ee
Let $s_0>1$ and consider the rescaled time 
\be
\label{rescaledtime}
{s}=s(t)=  s_0+ \int_{0}^{t} \frac {dt'}{{\lambda}^3(t')} \quad 
\text{or equivalently}\quad
\frac {d{s}}{dt} = \frac 1{{\lambda}^3},\quad s(0)=s_0.
\ee 
Then, on $[s_0,s(t_0)]$,
\begin{align}
& \left|\frac {{\lambda}_{s}}{\lambda} + {b}\right|+\left| \frac { x_{s}}{\lambda} - 1 \right|  
  \lesssim  
\left(\int \varepsilon^2 {e^{-\frac{|y|}{10}}} \right)^{\frac 12} +  b^2
  +p^2
 + \frac \l x |p|,\label{eq:2002bis} \\
& |{b}_{s}|  \lesssim
 \int  \varepsilon^2   {e^{-\frac{|y|}{10}}}+|p| \left( \int \varepsilon^2 {e^{-\frac{|y|}{10}}} \right)^{\frac 12}+|b|^2+  |b| |p| +p^4
 + \frac \l x |p| ,\label{eq:2003}\\
 & \left| \frac d{ds} \left(\frac b{\l^2} + \frac 4{\left(\int Q\right)}c_0     \l^{-\frac 32}  x^{-\theta} \right)   \right|\label{fond} \\ & \lesssim  \frac 1{\l^2} 
\left(  |b|^3+|p|^3 +   (|b|+|p|)\left(\int \varepsilon^2 {e^{-\frac{|y|}{10}}} \right)^{\frac 12}
+ \int \varepsilon^2 {e^{-\frac{|y|}{10}}} +  \frac {\l^2}{x^2} |p| +  \frac {\l}{x^3} |p|   \right).
\nonumber
\end{align}
\end{proposition}
\begin{remark}
The bounds \eqref{eq:2002bis}-\eqref{fond} will justify the dynamical system \fref{dynsyst}.
\end{remark}
  \begin{proof}
  
  {\bf step 1} Proof of  (i). This is a standard modulation claim. As usual, the decomposition is first performed for a fixed time $t$.
  For $t\in [0,t_0]$ fixed, define the map
  $$
 \Theta: (\bar b,\bar \l,\bar x,w,z_0) \mapsto \left( \int Q\bar \e  , \int \Lambda Q \bar \e  ,\int Q' \bar \e  \right),
  $$
  where $x_0=1/z_0$,
  $$
  \bar \e(y)=\e_{(\bar b,{\bar \l},\bar x,w_1,z_0)}(y) =  {\bar \l}^{\frac 12}   w_1\left(t, {\bar \l} y + \bar x\right)  -  {\bar \l}^{\frac 12} {\l_1^{\frac 12}(t)} q_0\left(t,x_1(t) + \bar x\right)\Yz(y)- Q_{\bar b}(y),
  $$
  $$
  w_1(t,y)= { \l_1(t)}^{\frac 12}   w \left(t, {\l_1(t)} y + x_1(t)\right)=Q(y)+\e_1(t,y).
  $$
 We have $\bar \e|_{(0,1,0,Q,0)}=0,$ so that $\Theta(0,1,0,Q,0)=0$ and
 $$
  \quad \partial_{\bar b} \bar \e |_{(0,1,0,Q,0)}= P,
 \quad \partial_{\bar \l} \bar \e|_{(0,1,0,Q,0)}= \Lambda Q,
 \quad \partial_{\bar x} \bar \e|_{(0,1,0,Q,0)}= Q'.
 $$
 so that  differentiating the map $\Theta$ with respect to the variables $(\bar b,\bar \l,\bar x)$
 at the point $(0,1,0,Q,0)$ we find the Jacobian matrix
 $$
 \left(\begin{array}{ ccc }  (P,Q) & (P,\Lambda Q) & (P,Q') \\ (\Lambda Q,Q) & (\Lambda Q,\Lambda Q) & (\Lambda Q,Q') \\   (Q',Q) & (Q',\lambda Q)  & (Q',Q')\\  \end{array}\right)=
  \left(\begin{array}{ ccc }  (P,Q) & (P,\Lambda Q) & 0\\0 & (\Lambda Q,\Lambda Q) &0 \\ 0 & 0  & (Q',Q')\\  \end{array}\right),
 $$
 which is not degenerate since $(P,Q)>0$. It follows from these observations that we can apply the implicit function theorem to $\Theta$: for $w_1$ small and $x_0$ large, there exists
 a unique $(\bar b ,\bar \l , \bar x)=(\bar b ,\bar \l , \bar x)(w_1,x_0)$ close to $(0,1,0)$ such that $\Theta(\bar b,\bar \l,\bar x,w_1,\tfrac 1{x_0})=0$. Then, we  define $b(t)=\bar b(w_1(t),x_0)$, $\l(t)=\bar \l(w_1(t),x_0) \l_1(t)$, $x(t)=\bar x(w_1(t),x_0)+x_1(t)$
 and $\e(t)=\bar \e(t)$.
 The regularity of $(b(t),\l(t),x(t))$   now follow
 from standard arguments.
 It follows  that we have the following decomposition of $u(t,x)$:
\begin{align}
u(t,x) & =  \frac 1{\l^{\frac 12}(t)} \left( Q_{b(t)} + p(t) \Yz+ \e   \right) \left( t,\frac {x-x(t)}{\l(t)}\right)
+q_0(t,x) \label{dec:2}\\
& = \frac 1{\l^{\frac 12}(t)} \left( Q_{b(t)} + p(t) \Yz+ \e + q \right) \left(t, \frac {x-x(t)}{\l(t)}\right)\label{dec:1}.
\end{align}

{\bf step 2} Equation of $\e$ and a priori bounds. To write the equation of $\e$, we first derive the equation of $w$ from the equations of $u(t)$ and $q_0(t)$:
 \be\label{eqw}
 w_t + (w_{xx} + w^5)_x = - (W_0)_x,
 \ee
 where
 \be\label{Wz}
 W_0 = 5 w^4 q_0 + 10 w^3 q_0^2 + 10w^2 q_0^3 + 5w q_0^4.
  \ee
  Second,   set 
  $ 
 \e_Y(s,y)= p(s) \Yz(y) +\e(s,y)$ so that
 $$
 w(s,x) = \frac{1}{\l^{\frac 12}(s)} \left( Q_{b(s)} +
 \e_Y \right)\left(s,\frac {x-x(s)}{\l(s)}\right).
  $$
By standard computations, we obtain for $\e_Y$:
  \begin{align}  \nonumber  
  \partial_s \e_Y & = (-\partial_{y}^2 \e_Y + \e_Y - (\e_Y+Q_b)^5 +Q_b^5)_y
- (5 Q^4 q)_y +\lsl \Lambda \e_Y
\\
& + \left(\frac {{\lambda}_{s}}{{\lambda}}+{b}\right) {\Lambda} Q_b 
+ \left(\frac { x_{s}}{\lambda} -1\right) (Q_b + \e_Y)_y   
+ \Phi_{b}  
 + \Psi_{{b}}   -W_y,
\label{eqofepsY}\end{align}
where  $$
W=5(Q_b+ \e_Y)^4 q- 5 Q^4 q + 10(Q_b+\e_Y)^3 q^2 + 10(Q_b+\e_Y)^2 q^3 + 5(Q_b+\e_Y) q^4.
$$
Finally, we replace   $ 
 \e(s,y) = \e_Y(s,y)- p(s) \Yz(y) $ and use $L \Yz = 5 Q^4  $, to obtain
   \begin{align}  \nonumber  
  \partial_s \varepsilon  &  = \left(-\partial_{y}^2 \e  + \e  - (\e +p \Yz +Q_b)^5 +Q_b^5 + p5Q^4 \Yz\right)_y
\\ & -p_s \Yz +(  5 Q^4 (p-q) )_y +\lsl (\Lambda \e+p\Lambda \Yz)
 + \left(\frac {{\lambda}_{s}}{{\lambda}}+{b}\right) {\Lambda} Q_b 
\nonumber \\
&+ \left(\frac { x_{s}}{\lambda} -1\right) (Q_b + \varepsilon+p  \Yz)_y   
+ \Phi_{b}  
 + \Psi_{{b}}   -W_y,
\label{eqofeps}\end{align}
and
\begin{align*}
W&=5(Q_b+ \e+p  \Yz)^4 q- 5 Q^4 q + 10(Q_b+\e+p  \Yz)^3 q^2 \\
&+ 10(Q_b+\e+p  \Yz)^2 q^3 + 5(Q_b+\e+p  \Yz) q^4.
\end{align*}
We now claim the following bounds which we will be used along the proof:

\begin{claim}\label{cl:one}
{\rm (a) Estimates on $q(s)$.}
\be\label{L2H1q}
\|q(s)\|_{L^2} \lesssim x_0^{-\theta+\frac 12}, \quad
\|q_y(s)\|_{L^2}\lesssim \l(s) x_0^{-\theta-\frac 12}, \quad
\|q(s)\|_{L^\infty}\lesssim \l^{\frac 12} (s) x_0^{-\theta}.
\ee
 {\rm (b) Properties of the function $p(s)$:}
\be\label{cl.1bis}
\left|p(s) - c_0 \l^{\frac 12}(s) x^{-\theta}(s)\right| \lesssim
c_0 \l^{\frac 12}(s) x^{-\theta-2}(s)\lesssim  x^{-2}(s) |p(s)|,
\ee
\be\label{cl.2}
 e^{-\frac{3|y|}{4}} |p(s)-q(s,y)|     \lesssim \frac {\l(s)}{x(s)} |p(s)| e^{-\frac {|y|} 4} 
,
\ee
\be\label{cl.4bis}
\left| \left( (5Q^4(p-q))_y ,Q\right) - c_0 \left( \int Q\right) \theta \l^{-\frac 32} x^{-\theta-1} \right| 
\lesssim  \frac {\l^2(s)}{x^2(s)} |p(s)|+\frac {\l(s)}{x^3(s)} |p(s)|,
\ee
\be\label{cl.2bis}
\left| p_s - \frac 12   \lsl p + \theta \frac {x_s} x p \right| \leq 
\frac {\l}{x^{3}}|p|   \left(\left|\xsl - 1\right|  + 1 \right),
\ee
\be\label{cl.2tri}
|p_s| \lesssim  \left| \lsl\right| |p| + \frac {\l}{x} |p| \left(\left|\xsl - 1\right|  + 1 \right).
\ee
{\rm (c) Estimates for the remainder term $W$:} Let 
\begin{align*}
\widetilde W&=5(Q_b+ \e+p  \Yz)^4 p- 5 Q^4 p + 10(Q_b+\e+p  \Yz)^3 p^2 \\
&+ 10(Q_b+\e+p  \Yz)^2 p^3 + 5(Q_b+\e+p  \Yz) p^4.
\end{align*}
Then,
\be\label{cl.3}
\int |W-\widetilde W| e^{-\frac 34 {|y|}}  \lesssim
\left(  |p|+|b| + \left( \int \e^2 e^{-\frac {|y|}{10}} \right)^{\frac 12}\right)  \frac {\l}{x} |p| ,
\ee
\be\label{cl.4}
|((\widetilde W)_y,\Lambda Q)|+|((\widetilde W)_y,Q)|\lesssim 
 b^2 + p^8 +|b| |p|  +|p| \left(\int \e^2 e^{-\frac {|y|}{10}}\right)^{\frac 12} +\int \e^2 e^{-\frac {|y|}{10}}  ,
\ee
\be\label{cl.5}
|((\widetilde W)_y,y \Lambda Q)| \lesssim 
\int \e^2 e^{-\frac {|y|}{10}} + b^2 + p^2.
\ee
\end{claim}

\begin{proof}[Proof of Claim \ref{cl:one}]
{\it Proof of} (a). Since $q_0(t)$ is solution of \eqref{kdv}, for all $t$,
$$
\|q_0(t)\|_{L^2}= \|f_0\|_{L^2} \lesssim x_0^{-\theta+\frac 12},\quad
E(q_0(t)) = E(f_0) \lesssim x_0^{-\theta- \frac 12},
$$
and  \eqref{L2H1q} follows.\\

{\it Proof of} (b). Since $p(s) = \lambda^{\frac 12}(s) q_0(s,x(s))$, 
\eqref{cl.1bis} follows from \eqref{surX} and \eqref{leqz3}. In particular, since $c_0<0$, $p<0$.

Next, by \eqref{surX}, \eqref{leqz3} and \eqref{estfz},  splitting the two cases $\l |y|<\frac {x(s)}{4}$
and $\l |y| > \frac {x(s)}{4}$, we have
\begin{align}
&  e^{-\frac{3|y|}{4}} |p(s)-q(s,y)|  =
e^{-\frac{3|y|}{4}} \l^{\frac 12}(s) |q_0(s,x(s))-q_0(s,\l(s) y +x(s))| 
\nonumber\\
&\lesssim  |y| e^{-  \frac {|y|} 2 } \l^{\frac 32}(s) \left(  e^{- \frac { |y|} 4}  \|q_0'(s)\|_{L^\infty(x>\frac 34 x(s))}
+   e^{-\frac {x(s)}{16\l(s)}}\|q_0'(s)\|_{L^\infty} \right)\nonumber \\
&   \lesssim c_0 e^{-\frac {|y|} 4}  \l^{\frac 32}(s)  \left(      x^{-\theta-1}(s)
+   e^{-\frac {x(s)}{16\l(s)}} x_0^{-\theta-1} \right)\nonumber \\
&\lesssim c_0 e^{-\frac {|y|} 4}  \l^{\frac 32}(s)        x^{-\theta-1}(s)  
 \lesssim \frac {\l}{x} |p| e^{-\frac {|y|} 4}.\label{pmq}\end{align}
Next, by \eqref{surX} and \eqref{leqz3},
\begin{align*}
&\left( (5Q^4(p-q))_y ,Q\right) = -5 \int Q^4 Q_y  (p-q)= -\int Q^5 q_y 
\\
&= -\l^{\frac 32}(s) \int Q^5(y) \partial_x q_0(s,\l(s) y +x(s)) dy\\
&= - \l^{\frac 32}  f_0'( x(s)) \int Q^5  - \l^{\frac 32} \int Q^5 \left(f_0'(\l y + x(s)) -f_0'(x(s))\right)dy  \\ & -\l^{\frac 32}(s) \int Q^5(y) \left( \partial_x q_0(s,\l(s) y +x(s) )-f_0'(\l y+x(s))\right) dy\\
&
 =c_0 \left( \int Q\right) \theta   \l^{\frac 32}(s)  x^{-\theta -1}(s) 
 + O\left( \l^{\frac 32}e^{-\frac 54   x(s) } \right) + O\left( \l^{\frac 52} x^{-\theta-2}\right) 
+O\left( \l^{\frac 32} x^{-\theta-3}(s)\right)\\
&
 =c_0 \left( \int Q\right) \theta   \l^{\frac 32}(s)  x^{-\theta -1}(s) 
 + O\left(\frac {\l^2}{x^2} |p|  \right)
 + O\left(\frac {\l}{x^3} |p|  \right),
 \end{align*}
where we have split the integrals above  into $|y| >\frac 14 x(s)$ and $|y|<\frac 14 x(s)$ and using the fact that for $|y|<\frac 14 x(s)$, $\l y + x(s) >x(s)-|y|> \frac 34 x(s)   > \frac 12 t + \frac 12 x_0$, so that
 \eqref{leqz3} holds for $x= \l y + x(s)$, and \fref{cl.4bis} is proved.\\
Now, we prove \eqref{cl.2bis}. By
explicit differentiation and Lemma \ref{le:qz},
\begin{align*}
p_s - \frac 12 \lsl p & = 
\l^{\frac 12}(q_0)_s(s,x(s)) + \l^{\frac 12} x_s \partial_x q_0(s,x(s))\\
& = \l^{\frac 72} (q_0)_t(s,x(s)) + \l^{\frac 12} x_s f_0'(x(s))   
+ O\left(  \frac {\l}{x^{3}} |p| \left(\left|\xsl - 1\right|  + 1 \right) \right)\\
& = -\theta \frac {x_s} x |p| + O\left(  \frac {\l}{x^{3}} |p| \left(\left|\xsl - 1\right|  + 1 \right) \right).
\end{align*} 
Since $|x_s| \leq \l (|\xsl -1| + 1)$, \eqref{cl.2tri} follows.\\

{\it Proof of} (c). For \eqref{cl.3}, we first note
\begin{align*}
\widetilde W - W &=5\left[(Q_b+ \e+p  \Yz)^4 -   Q^4\right]  (p-q) + 10(Q_b+\e+p  \Yz)^3 (p^2-q^2) \\
&+ 10(Q_b+\e+p  \Yz)^2 (p^3-q^3) + 5(Q_b+\e+p  \Yz) (p^4-q^4).
\end{align*}

Thus, 
 \begin{align*}
&  \int e^{-\frac 34 {|y|}} |W-\widetilde W|  \\ &\lesssim
\left(  |p|+|b| \right) \int   e^{- \frac 34 {|y|}}  |p-q|
+   \left( \int \e^2 e^{-\frac {|y|}{10}} \right)^{\frac 12}\left( \int   e^{-(\frac 32-\frac 1{10}) {|y|}}  |p-q|^2\right)^{\frac 12}  \\ & + \int   e^{-\frac 34 {|y|}} |p^2-q^2|   \lesssim
\left(  |p|+|b| + \left( \int \e^2 e^{-\frac {|y|}{10}} \right)^{\frac 12}\right)  \frac {\l}{x}|p|,
\end{align*}
 using \eqref{cl.2}, and the following similar estimate
$$
e^{-\frac 34 {|y|}} |p^2-q^2| \lesssim \frac {\l}{x} p^2 e^{-\frac {|y|}{4}}.
$$
Next, \eqref{cl.4} follows from the parity properties and then direct estimates. \eqref{cl.5} follows from direct estimates. Note that   $p^2$   appears in \eqref{cl.5} because there is no cancellation due to parity for this term. This concludes the proof of Claim \ref{cl:one}.
\end{proof}

{\bf step 3}  Estimates induced by the conservation laws.
By $L^2$ norm conservation,
\begin{align*}
& \int u^2(0) - \int Q^2  = \int Q_b^2 - \int Q^2 + \int (\e+p\Yz+q)^2 + 2 \int (\e+p\Yz +q) Q_b\\
&= 2 b(P,Q) + O(|b|^{2-\gamma}) +\|\e\|_{L^2}^2 + O(|b|^{1-\frac \gamma 2}\|\e\|_{L^2}) +O(|p|+   \|q\|_{L^2}).
\end{align*}
Estimate \eqref{twobound} follows. By energy conservation, $Q''+Q^5= Q$  and $\int \e Q=0$,
\begin{align*} & 2 \l^2 E(u_0)   = 2 E(Q_b + \e + p\Yz +q) \\& = 2 E(Q_b) - 2 \int (\e+p\Yz + q) \left( (Q_b-Q)_{yy}+ (Q_b^5 - Q^5) \right)
- 2 \int (p\Yz + q) Q \\ &+ \int (\e+p\Yz + q)_y^2 - \frac 13 \int \left( (Q_b+\e+p\Yz + q)^6 - Q_b^6 - 6 Q_b^5(\e+pY+q)\right)\\
& = - 2 b (P,Q) + O(b^2) 
+ O\left( |b|^{1+\frac 32 \gamma} \left( \|\e\|_{L^2} + |p| + \|q\|_{L^2}\right)\right) - 2 p \left( \int \Yz Q + \int Q\right)
\\ &+ 2 \int (p-q) Q + \|\e_y\|_{L^2}^2  + O(p^2) + O\left(\|\e_y\|_{L^2}(|p|+ \|q_y\|_{L^2})\right) + 2\l^2 E(q_0)
\\ &- \frac 13 \int \left( (Q_b+\e+p\Yz+q)^6 - Q_b^6 - 6Q_b^5 (\e+p\Yz+q) -q^6\right)
\end{align*}
By \eqref{YQ}, we have $\int \Yz Q + \int Q = \frac 14 \int Q$. Using in addition \eqref{PQ}, \eqref{cl.2}, \eqref{cl.1bis}, we obtain
\begin{align*}
\frac 1{\l^2}\|\e_y\|_{L^2}^2 & \lesssim |E(u_0)| + \left| \frac 14 \frac b{\l^2} \int Q + \frac p{\l^2}\right| 
+ \frac {b^2}{\l^2}+  \frac {|b|^{1+\frac 32 \gamma}}{\l^2} \left( \|\e\|_{L^2} + |p| + \|q\|_{L^2}\right) \\
 & + \frac { |p| }{\l x}  + \frac {p^2}{\l^2}  +  x_0^{-\theta-\frac 12}\\
 & \lesssim |E(u_0)| + \left| \frac 14 \frac b{\l^2} \int Q + c_0 \l^{-\frac 32} x^{-\theta} \right|  
+ \frac {b^2}{\l^2}+  
 \frac { |p| }{x^2} + \frac { |p| }{\l x}  + \frac {p^2}{\l^2}  +  x_0^{-\theta-\frac 12} 
  .
\end{align*}

\medskip

{\bf step 4} Modulation equations. We argue as in \cite{MMR1}, proof of Lemma 2.7, differentiating with respect to $s$ the orthogonality conditions $\int \e \Lambda Q=0$, $\int \e Q'=0$ and 
$\int \e Q=0$ and using \eqref{eqofeps} to obtain   \eqref{eq:2002bis}
and \eqref{eq:2003}.
Here, we  will  treat only  the   terms  coming from $q$ and $p\Yz$ in \eqref{eqofeps} and we refer the reader to \cite{MMR1} for more details on the other terms.
\medskip

\emph{Proof of   \eqref{eq:2002bis} and \eqref{eq:2003}.} 
It follows from computations in \cite{MMR1} proof of Lemma~2.7 and Claim \ref{cl:one} that
\begin{align*}
& \left| \left( {\frac{{\lambda}_s}{{\lambda}}} + {b} \right) -  \frac {( \e,L(\Lambda Q)')}{\|\Lambda Q\|_{L^2}^2}  \right|  
 \\ & \lesssim 
 \left( \left|  {\frac{{\lambda}_s}{{\lambda}}}+ {b}  \right|  +|b|\right)\left(|p|+ |b|  + \left( \int \varepsilon^2 {e^{-\frac{|y|}{10}}} \right)^{\frac 12} \right)  + |p| \left( \int \varepsilon^2 {e^{-\frac{|y|}{10}}} \right)^{\frac 12} \\
 & + \left| {\frac{x_s}{{\lambda}}} - 1 \right| \left(  |b|  + \left( \int \varepsilon^2 {e^{-\frac{|y|}{10}}} \right)^{\frac 12} \right)
 +|p_s|+|{b}_{s}|+\int \varepsilon^2 e^{-\frac{|y|}{10}} + p^8    +
 \frac {\l}{x} |p|   .
\end{align*}

We proceed similarly for $\xsl-1$, taking into account different cancellations due to parity properties \begin{align*}
&  \left|\left( {\frac{x_s}{{\lambda}}} - 1  \right) - \frac {( \e,L(y \Lambda Q)')}{\|\Lambda Q\|_{L^2}^2}  \right|
 \\ & \lesssim 
 \left( \left|  {\frac{{\lambda}_s}{{\lambda}}}+ {b}  \right| + \left| {\frac{x_s}{{\lambda}}} - 1 \right|  +|b|\right)\left(|p|+|b|  + \left( \int \varepsilon^2 {e^{-\frac{|y|}{10}}} \right)^{\frac 12} \right) 
 \\ &+|{b}_{s}|+\int \varepsilon^2 e^{-\frac{|y|}{10}} 
   +p^2 +\frac \l x |p|.
\end{align*}
 
Then, taking the scalar product of \eqref{eqofeps} by $Q$ and arguing similarly, we have the following 
rough estimate for $b_s$:
\begin{align*}
|{b}_{s}| &\lesssim  \left| {\frac{{\lambda}_s}{{\lambda}}} + {b}\right|^2 + \left| {\frac{x_s}{{\lambda}}} - 1 \right|^2 +  {|b|}^2 +  \int   \varepsilon^2   {e^{-\frac{|y|}{10}}} 
+ |p| \left( \int \varepsilon^2 {e^{-\frac{|y|}{10}}} \right)^{\frac 12}
\\ &+|p_s| + |b| |p| +p^8
 +\frac \l x |p| .
\end{align*}

Combining these estimates with (from \eqref{cl.2bis})
$$
|p_s| \leq \left|   \lsl \right| |p| + \frac {\l}{x}  |p|\left( \left|\xsl - 1\right| + 1 \right)
\leq |b| |p| +\left| \lsl+b\right| |p| +  \frac \l x |p|\left( \left|\xsl - 1\right| + 1 \right),
$$
from \eqref{cl.2tri},
we obtain  \eqref{eq:2002bis} and \eqref{eq:2003}. 

\medskip

\emph{Proof of \eqref{fond}.}
First, we derive a refined equation for $b_s$, taking the scalar product of equation \eqref{eqofeps} by $Q$ and proceeding as   in \cite{MMR1},  proof of Lemma 2.7.

Recall from \cite{MMR1},
$$(\Psi_b,Q)= -\frac {b^2}{8} \|Q\|_{L^1}^2 + O(|b|^3),\quad  (\Phi_b, Q) = -\frac {b_s}{16} \|Q\|_{L^1}^2 +O(|b|^{10}).
$$
Note also that from direct computations and parity properties
\begin{align*}
&\left|\left( ((\e+p\Yz + Q_b)^5 - Q_b^5 - p5Q^4 \Yz - 5Q^4 \e)_y,Q\right)
- 20 bp((Q^3\Yz P )_y,Q) \right| \\
&+ \left| \left((\widetilde W)_y,Q\right)-20 bp (( P Q^3)_y,Q)\right|\\
& \lesssim  |b|^3 +|p|^3 + \int \varepsilon^2 {e^{-\frac{|y|}{10}}} 
+ (|p|+|b|) \left(\int \varepsilon^2 {e^{-\frac{|y|}{10}}} \right)^{\frac 12} 
  .
\end{align*}
(See \cite{MMR1} for details on the nonlinear terms in $\e$.)
Using   \eqref{cl.4bis} and 
the above estimates
we find\begin{align}
&\bigg| b_s + 2b^2  - \frac {16}{(\int Q)^2} \bigg( -  p_s (\Yz,Q)  - \lsl p 
 (\Yz,\Lambda Q)  \nonumber\\
&  -  bp \left[ 20 ((Q^3\Yz P)_y,Q) + 20 ((PQ^3)_y,Q) \right]
 + c_0 \theta \left(\int Q\right) \l^{\frac 32}(s) x^{-\theta-1}(s) \bigg)\bigg|\\
& \lesssim |b|^3+  |p|^3 +\int \varepsilon^2 {e^{-\frac{|y|}{10}}} +(|p|+|b|) \left(\int \varepsilon^2 {e^{-\frac{|y|}{10}}} \right)^{\frac 12} 
 +\frac {\l^2}{x^2} |p|+\frac {\l}{x^3} |p|.\label{eq:bautre}
\end{align}

We claim  the following cancellation
\be\label{cancel}
- (\Yz,\Lambda Q) + 20 ((Q^3\Yz P)_y,Q) + 20 ((PQ^3)_y,Q) =0.
\ee
Indeed, 
$L(P')= (LP)' + 20 Q^3 Q' P = \Lambda Q + 20 Q^3 Q' P,$
and so from \eqref{eq:23},
\begin{align*}
& - (\Yz,\Lambda Q) + 20 ((Q^3\Yz P)_y,Q) + 20 ((PQ^3)_y,Q) 
\\ &= - (\Yz,\Lambda Q)-  20 (\Yz+1,PQ^3 Q')\\
& = - (\Yz+1 , L(P')) + \int \Lambda Q = - (L(\Yz+1),P') - \frac 12 \int Q
\\ & = - \int P' -  \frac 12 \int Q  =  P(-\infty) - \frac 12 \int Q =0.
\end{align*}
Thus,
\begin{align*}
& \left| 
 - \lsl p 
 (\Yz,\Lambda Q) +  bp \left[ 20 ((Q^3\Yz P)_y,Q) + 20 ((PQ^3)_y,Q) \right]
\right|  = 
|p| \left| \lsl+b \right|  |(\Yz,\Lambda Q)|.
\end{align*}

\medskip

Now, from $(\Yz,Q) =- \frac 34 \int Q$ (see   \eqref{YQ}),
  using \eqref{cl.1bis} and \eqref{cl.2bis} we note that 
\begin{align}
& \frac {16}{(\int Q)^2} \bigg( -  p_s (\Yz,Q)+ c_0 \theta \left(\int Q\right) \l^{\frac 32}(s) x^{-\theta-1}(s) \bigg)\nonumber \\
& =   \frac {12}{(\int Q)} c_0 \l^{\frac 12} x^{-\theta}  \left(  \frac 12 \lsl 
 +\frac \theta 3    \frac {x_s}{x}    \right) +O\left(\frac {\l}{x^{3}}| p|\right)  + O\left( \frac {\l}{x} |p| \left|\xsl-1\right|\right).
    \label{pascancel}
\end{align}
Therefore, 
\begin{align}
&\bigg| b_s + 2b^2 - \frac {4}{(\int Q)} c_0 \l^{\frac 12} x^{-\theta}  \left(  \frac 32 \lsl 
 +  \theta     \frac {x_s}{x}    \right)   \bigg| \nonumber \\
& \lesssim  |b|^3 +|p|^3 + \int \varepsilon^2 {e^{-\frac{|y|}{10}}} 
+ (|p|+|b|) \left(\int \varepsilon^2 {e^{-\frac{|y|}{10}}} \right)^{\frac 12} +
\frac {\l^2}{x^2} |p| + \frac {\l}{x^3} |p|
.\label{eq:b}
\end{align}

Now, we prove \eqref{fond}. By direct computation,
$$
\frac d{ds}\left( \frac 4{\left(\int Q\right)}c_0 \theta   \l^{-\frac 32}(s) x^{-\theta}\right) = - \frac {4}{(\int Q)} c_0 \l^{-\frac 32} x^{-\theta}  \left(  \frac 32 \lsl 
 +  \theta     \frac {x_s}{x}    \right)  
$$
$$
\frac d{ds}\left( \frac b{\l^2}\right)
 =   \frac {b_s}{\l^2} - 2 \lsl \frac b{\l^2}
 =   \frac {b_s}{\l^2}  + 2 \frac {b^2}{\l^2} - 2 \left( \lsl +b\right) \frac b{\l^2}.$$
 and \eqref{fond} follows from \eqref{eq:b} and \eqref{eq:2002bis}.
  \end{proof}

 
\section{Proof of Theorem \ref{th:1}}


This section is devoted to the proof of Theorem \ref{th:1} which follows from the modulation equations of Proposition \ref{le:2} coupled with the control of the well localized error $\e$ as in \cite{MMR1}. We present the   new dynamical arguments and report the proofs of  two technical Lemmas adapted from \cite{MMR1} to the Appendices.

\subsection{The bootstrap argument}

Let $$
\beta =   \frac { 2(\theta-1)} {2\theta-1},\quad
\theta = \frac {1-\frac \beta2}{1-\beta}, \quad 0< \beta< \frac {11}{20}, \quad
1<\theta <\frac {29}{18},$$ and define
\be
\label{defcobeta}
c_0=-\frac {\int Q}2(\theta-1)(2\theta-1)^{\theta-1}.
\ee
Given $s>s_0$, $(b(s),\l(s),x(s))\in \RR^*_+\times \RR^*_+\times \RR$, we define:
 $$
g(s) =\frac {b(s)}{\l^2(s)} + \frac 4{\int Q} c_0   \l^{-\frac 32}(s) x^{-\theta}(s),\quad
f(s) =   \l^{\frac 12} (s) + \frac 2 {\int Q} c_0 \frac {1}{\theta-1} {x^{-\theta+1}(s)}.
$$
Let  $(\varphi_i)_{i=1,2},\psi$  be smooth functions such that:
\bea
\label{defphi1bis}
\varphi_i(y) =\left\{\begin{array}{lll}e^{y}\ \ \mbox{for} \ \   y<-1,\\
 1+y  \ \ \mbox{for}\ \ -\frac 12<y<\frac 12,\\ 
 y^i\ \ \mbox{for}\ \ \mbox{for}\ \ y>2,
 \end{array}\right.  \ \ \varphi_i'(y) >0, \ \ \forall y\in \RR,
\eea
\bea
\label{defphi2}
\psi(y) =\left\{\begin{array}{ll} e^{2y}\ \ \mbox{for}   \ \ y<-1,\\
 1  \ \ \mbox{for}\ \ y>-\tfrac 12,\end{array}\right.  \ \ \psi'(y) \geq 0 \ \ \forall y\in \RR.
\eea
Let $B> 100$ and
 $$\psi_B(y)=\psi\left(\frac yB\right), \ \ \varphi_{i,B}=\varphi_i\left(\frac yB\right), \ \ i=1,2,$$ 
 and define the following norms on $\e$
  \begin{equation}\label{eq:no}
{\mathcal{N}_i}  (s)  = \int  \varepsilon_y^2(s,y) \psi_{B}(y)dy + \int  \varepsilon^2(s,y) \varphi_{i,B}(y)   dy, 
\end{equation} 
\be\label{eq:noloc}
{\mathcal{N}_{i,\rm loc}} (s)  =    \int\varepsilon^2 (s,y)     \varphi'_{i,B} (y) dy,\quad i=1,2.
\ee
\\
We now claim the following bootstrap Proposition which is the heart of the analysis:

\begin{proposition}[Bootstrap]
\label{PR:BS}
Let $s_0=s_0(\beta)>1$ large enough and set
\be
\label{initixzero}
x_0= x (s_0)=  \frac 1{(1-\beta)}{s_0^{1-\beta}}.
\ee
Let  $\e_0\in H^1$ be such that 
\be\label{A2}
s_0^{10}\left[ \int_{y>0} y^{10} \e^2_0(y) dy+\|\e_0\|^2_{H^1} \right]<1, \quad (\e_0, y {\Lambda}   Q )=(\e_0,\Lambda Q) =(\e_0,Q)=0.\ee
Then, there exists  
\be\label{estimation2}
(\l_0,b_0)\in\mathcal D =\left \{(\l,b) \ : \ | \l - s_0^{-\beta}| \leq s_0^{-\beta - \frac 1{10}},
\ |b  - \beta s_0^{-1}| \leq s_0^{-1-\frac 1{10}}\right\},
\ee 
such that the solution of \eqref{kdv} with initial data
\be\label{defu0}
u_0(x)   =  \frac 1{\l^{\frac 12}_0} \left( Q_{b_0} + \l^{\frac 12}_0 q_0(s_0,x_0) \Yz+ \e_0   \right) \left( \frac {x-x_0}{\l_0}\right)
+q_0(s_0,x)
\ee
has a decomposition $(b(s),\l(s),x(s), \e(s))$ as in Proposition \ref{le:2} which satisfies\footnote{recall that $s=s(t)$ is the rescaled time \fref{rescaledtime}.} on $[s_0,+\infty)$:
$$\leqno{\rm (BS1)}\quad
\left(|g(s)|  s^{1 - 2 \beta + \frac 15} \right)^2+ 
\left(|f(s)| s^{\frac \beta 2 + \frac 1{10}}\right)^2\leq 1; 
$$
$$\leqno{\rm (BS2)}\quad 
|b(s)|\leq   {10} s^{-1},
\
\frac 1{10} s^{-\beta}\leq \l(s) \leq 10 s^{-\beta},
\  
\frac 1{10} s^{1-\beta} \leq (1-\beta) x(s) \leq 10 s^{1-\beta};
$$
$$\leqno{\rm (BS3)}\quad
\int_{y>0}y^{10}\e^2(s,y)dy\leq 10 \l^{-10} ,\quad 
\mathcal N_{i}(s) \leq s^{-\frac 52},\quad \|\e(s)\|_{H^1} \lesssim \delta(\alpha^*).
$$
Moreover,
\be\label{sharp}
\left|  \frac {(1-\beta)}{s^{1-\beta}}  x (s) - 1\right|
+\left|  s^{\beta} \l(s) -1\right|
+\left| \frac {s}{\beta } b(s) -1\right|  \lesssim s^{-\frac 1{10}}, 
\ee
\be\label{sharpe}
\l^{-1} \|\e_y(s)\|_{L^2}+ \|\e(s)\|_{L^2} \lesssim 1.
\ee
\end{proposition}
 Let us observe that \fref{sharp} now gives the leading order behavior of the scaling parameter $\lambda(s)=\frac{1}{s^{\beta}}(1+o(1))$, and the conclusion of Theorem \ref{th:1} now immediately follow from the change of variables \fref{rescaledtime} depending on the value of $\beta$ as in step (iv) of section \ref{vneovnoen}.\\

The rest of this section is therefore devoted to the proof of Proposition \ref{PR:BS}. First observe by uniqueness of the decomposition that
$$
b_0=b(s_0),\quad  \l(s_0)=\l_0 \quad x(s_0)=x_0, \quad \e(s_0)=\e_0.
$$
We now argue by contradiction, assuming that for all 
$(\l_0,b_0)\in \mathcal D$,    we have
$$
s^*(\l_0,b_0) := \sup \left\{
s\geq s_0 \hbox{ such that   (BS1)-(BS2)-(BS3) holds on $[s_0,s]$}
\right\} <+\infty.
$$
We will derive a contradiction by first closing the bootstrap bounds (BS1)-(BS2)-(BS3), and then finding a couple $(\l_0,b_0)$ using a topological argument.


\subsection{First consequences of the bootstrap bounds} 


Let us start with some quantitative bounds which follow from the bootstrap bounds and Proposition \ref{le:2}.

\begin{claim}[Consequences of the bootstrap estimates]\label{le:3.1}
 
{\rm (i)} For $s_0=s_0(\beta)$ large enough, there holds:
\begin{itemize}
\item if $\beta>\frac 13$, for all $s\in (s_0,s^*)$, $t(s)= \int_{s_0}^{s} \l^3(s') ds'<1$;
\item if $0<\beta\leq \frac 13$, for all $t>0$, $x(t) \geq \frac 23 t + \frac 23 {x_0}$.
\end{itemize}
{\rm (ii)}
For all $s\in (s_0,s^*)$,
\be\label{BSp}
0< -p(s) \lesssim \frac 1 s,\quad \frac {\l}{x} \lesssim \frac 1 s,
\ee
\be\label{BSlsl}
\left|\lsl +b\right| + \left| \xsl - 1\right| \lesssim
\left(\int \varepsilon^2 {e^{-\frac{|y|}{10}}} \right)^{\frac 12} +  \frac 1{s^2},\ee
\be\label{BSparam}
\left| \lsl \right| \lesssim \frac 1 s + \left(\int \varepsilon^2 {e^{-\frac{|y|}{10}}} \right)^{\frac 12},\quad 
\left| x_s \right|  \lesssim  s^{-\beta}\left( 1+\left(\int \varepsilon^2 {e^{-\frac{|y|}{10}}} \right)^{\frac 12} \right)  ,
\ee
\be\label{BSbsps}
|b_s|+ |p_s| \lesssim   \int \varepsilon^2 {e^{-\frac{|y|}{10}}}    + \frac 1{s^2},
\ee
\begin{align}
& \left| \frac d{ds} \left(\frac b{\l^2} +\frac 4{\left(\int Q\right)}c_0     \l^{-\frac 32}(s) x^{-\theta}(s)\right)   \right| \nonumber \\ &\lesssim  \frac 1{\l^2} 
\left( \frac 1{s^{\frac {29}{10}}}   +  \frac 1 s \left(\int \varepsilon^2 {e^{-\frac{|y|}{10}}} \right)^{\frac 12}
+ \int \varepsilon^2 {e^{-\frac{|y|}{10}}}\right).\label{BSfond}
\end{align}
{\rm (iii)} For all $s\in (s_0,s^*)$,
\begin{align}
&\|\e(s)\|_{L^2}^2 \lesssim \left| \int u^2(0) - \int Q^2\right|+ s^{-1} + x_0^{-\theta+\frac 12},\label{BSeps}\\
& \frac 1{\l^2} \|\e_y(s)\|_{L^2}^2 \lesssim |E(u(0))|+ s^{- \frac 1{10}} + x_0^{-\theta-\frac 12} . 
\end{align}
 \end{claim}
  \begin{proof}  
 Let $\frac 13 <\beta<\frac {11}{20}$. Then
 $$
 t(s) = \int_{s_0}^{s} \l^3(s') ds' \leq 
 10^3 \int_{s_0}^{s} (s')^{-3 \beta} ds'  
 \leq \frac {10^3}  {3\beta-1} \left(s_0^{-3\beta+1}-s^{-3\beta +1}\right)
 <1
 $$
 for $s_0$ large enough. For $\beta = \frac 13$,
 $$
 t(s) = \int_{s_0}^{s} \l^3(s') ds' \leq 10^3\, \log\frac s{s_0} \quad 
 \hbox{so that} \quad s \geq s_0 e^{10^{-3} t}.
 $$
 Thus,
 $$
 x(s) \geq \frac 3{20} s^{\frac 23} \geq \frac 3{20} s_0^{\frac 23} e^{\frac 23 10^{-3} t}\geq 10 t^3 .
 $$
 for $s_0$ large enough.
 Since $x(s) \geq \frac 45 x_0$, we obtain $x(t) >  t^3 + \frac 12 x_0$. Finally, for $0<\beta <\frac 13$,
 $$
 t(s) = \int_{s_0}^{s} \l^3(s') ds' \leq \frac {10^3}  {1-3\beta} \left(s^{1-3\beta}-s_0^{1-3\beta}\right)
 $$
so that for $s_0$ large
$$s \geq \left( \frac {(1-3\beta) t}{10^3}+s_0^{1-3\beta}\right)^{\frac 1 {1-3\beta}} \geq 100 t^{\frac {1+\beta}{1-\beta}},
 $$
and
 $$x(s) \geq \frac 3{20} s^{1-\beta} \geq 15  t^{1+\beta}.$$
  Since $x(s) \geq \frac 45 x_0$, we obtain $x(t) >  t^{1+\beta} + \frac 23 x_0$.\\
  The estimate \eqref{BSp} is a consequence of \eqref{cl.1bis} and
 $\frac \beta 2 + (1-\beta) \theta = 1$, so that
 $$ 0< p \lesssim \l^{\frac 12 } x^{-\theta} (s) \lesssim
 s^{-\frac \beta 2} s^{-\theta  (1-\beta)} \lesssim \frac 1s.
 $$
 The estimates  \eqref{BSlsl}-\eqref{BSfond} are   immediate consequences  of \eqref{eq:2002bis}-\eqref{fond}, 
 \eqref{cl.2tri}, the bootstrap assumptions and the upper bound $\beta<\frac {11}{20}$.
 \end{proof}

\subsection{Closing the estimates on $\e$} We now close the bounds on $\e$ and claim the improved bound: for all $s\in [s_0,s^*]$,
 $$\leqno{\rm (BS3')}\quad
\int_{y>0}y^{10}\e^2(s,x)dx\leq 5 \l^{-10} ,\quad 
\mathcal N_{i}(s) \leq \frac 12 s^{-\frac 52},\quad \|\e(s)\|_{H^1} \lesssim \delta(\alpha^*).
$$
 Let $\varphi_{10}$ be a smooth function such that
$$\varphi_{10}(y)=\left\{\begin{array}{ll} 0\ \ \mbox{for}\ \ y\leq 0, \\ y^{10}\ \  \mbox{for}\ \ y\geq 1\end{array}\right., \quad 0\leq \varphi_{10}\lesssim \varphi_{10}' \hbox{ for $0<y<1$}.$$
The control of the tail of $\e$ on the right is a direct consequence of the following brute force monotoncity formula:

 \begin{lemma}[Dynamical control of the tail on the right]
\label{lemmatail}
For all $s\in [s_0,s^*]$,\be
\label{keyestimate}
 {\l^{-10}}\frac{d}{ds}\left\{\l^{10}\int\varphi_{10}\e^2\right\}\lesssim \mathcal N_{1,\rm loc}+\frac 1{s^2}.
\ee
\end{lemma}
See proof of Lemma \ref{lemmatail} in Appendix \ref{appA}.
\\
 
The  control of   $\mathcal N_i(\e)$ norm, which is fundamental for the proof, now follows by adapting the mixed Energy/Morawetz monotonicity formula first derived in \cite{MMR1}. Recall the definitions \fref{defphi1bis}, \fref{defphi2}, we claim:

\begin{lemma}[Monotonicity formula]
\label{propasymtp}
There exist $\mu>0$ such that the following holds for $B>100$ large enough.   Let   the   energy--virial Lyapounov functionals for $i=1,2$,  
\begin{align}
  {\cal F}_{i} & =  \int \Big[\varepsilon_y^2\psi_B + \varepsilon^2 \varphi_{i,B}  \label{feps}\\ &- \frac  13 \left((\varepsilon {+} Q_b{+}p\Yz{+}q)^6  -  (Q_b{+}p\Yz{+}q)^6   - 6 \varepsilon \left(Q_b^5{+}q^5{+}5Q^4(p\Yz{+}q)\right)\right)\psi_B\Big] .\nonumber
\end{align}
Then   the following estimates hold on $[s_0, s^*]$:\\
{\em (i)  Lyapounov control}: for $i=1,2$, $j\geq 0$
\be
\label{lyapounovconrolbis}
\frac{d}{ds}\left[ s^j {\mathcal F_{i}} \right]+      {\mu} s^j \int   \left(\varepsilon_y^2+\e^2\right)   \varphi'_{i,B}  \lesssim s^{j-4} +s^{j-9+10 \beta} .
\ee
{\em (iii) Coercivity of $\mathcal F_{i}$ and pointwise bounds}: for $i=1,2$, $j\geq 0$,
\begin{equation}
\label{lowerbound}
-\frac 1{s^4} + \mathcal N_i \lesssim \mathcal F_{i}\lesssim \frac 1{s^4} + \mathcal N_i .\end{equation}
\end{lemma}
See the proof of Lemma \ref{propasymtp} in Appendix \ref{appA}.
\\

\noindent \textbf{Proof of {\rm (BS3')}.}
From Lemma \ref{lemmatail}, (BS2) and (BS3), 
$$
 \frac d{ds} \left\{ \l^{10} \int \varphi_{10} \e^2 \right\}
\lesssim \l^{10}\left(\mathcal N_{1,\rm loc} + s^{-2}\right) 
\lesssim  s^{-10 \beta -2},
$$
so that by integration on $[s_0,s]$, and \eqref{A2},
$$
\l^{10}(s)\int \varphi_{10} \e^2(s) \leq
\l^{10}(s_0) \int \varphi_{10} \e^2(s_0) + Cs_0^{-1-10\beta}
\leq 2.
$$
By the properties of $\varphi_{10}$ and \eqref{A2}, we obtain
\be\label{aa3}
  \l^{10}(s) \int_{y>0}y^{10} \e^2(s,y) dy \leq 3.
\ee
Now, we apply Lemma \ref{propasymtp} with $j=\frac 52 $. We find by \eqref{lyapounovconrolbis}
and $\beta< \frac {11}{20}$,
\be
\label{neovneogbgeo}
\frac d{ds} \left[ s^{\frac 52} \mathcal F_i\right]
\lesssim s^{-\frac 32}  +
s^{-1-10 (\frac {11}{20}-\beta)}.
\ee
The initial smallness \fref{A2} ensures $$s_0^{10}|\mathcal F_i(0)|\lesssim 1$$ and thus the time integration of \fref{neovneogbgeo} on $[s_0,s^*]$ yields:
$$
\mathcal F_i (s)\lesssim s^{- \frac 52}  s_0^{-\delta}
$$
for some $\delta=\delta(\beta)>0$.
Using \eqref{lowerbound}, we conclude:
$$
\mathcal N_i(s) \leq s^{- \frac 52}    s_0^{-\delta} +s^{-4} \leq \frac 12 s^{- \frac 52}
$$
for $s_0$ large enough, which together with \eqref{aa3} and the control of the full $H^1$ norm through the conservation laws \fref{twobound}, \fref{energbound} concludes the proof of (BS3').


\subsection{Closing the estimates on  $(b,\l,x)$}

We now use the obtained bounds on $\e$ and the modulation equations on the geometrical parameters of Proposition \ref{le:2} to close the bounds on $(b,\l,x)$. We claim: for all $s\in [s_0,s^*]$,$$\leqno{\rm (BS2')}\quad 
|b(s)|\leq   {5} s^{-1},
\
\frac 1{5} s^{-\beta}\leq \l(s) \leq 5 s^{-\beta},
\  
\frac 1{5} s^{1-\beta} \leq (1-\beta) x(s) \leq 5 s^{1-\beta}.
$$
\\
\textbf{Proof of (BS2').}
First, note that from \eqref{fond}, \eqref{eq:2002bis}, and using (BS2)-(BS3), $\frac 13 \leq \beta\leq \frac 9{10}$, we have on 
$[s_0,s^*]$:
\be\label{dg}
\left| g_s \right| \lesssim s^{-\frac 9 4 + 2 \beta },
\ee
\be\label{dl}
\left| \lsl + b \right| \lesssim s^{-\frac 54},
\ee
\be\label{dx}
\left| \xsl -1 \right|\lesssim s^{-\frac 54}.
\ee

By (BS1) and (BS2), we have using $(1-\beta)(\theta-1) =\frac \beta 2$ the estimate:
\begin{align}
& \left| \l (s) -  \left(\frac 2 {\int Q} c_0 \frac {1}{\theta-1}\right)^2  {x^{-2 \theta+2}(s)} \right| \label{bbb2}\\ &=
\left| \l^{\frac 12} (s) + \frac 2 {\int Q} c_0 \frac {1}{\theta-1} {x^{-\theta+1}(s)}\right| \left|  \l^{\frac 12} (s) - \frac 2 {\int Q} c_0 \frac {1}{\theta-1} {x^{-\theta+1}(s)}\right|\nonumber \\
&\lesssim |f(s)| \left( \l^{\frac 12} (s)+x^{-\theta+1}(s)\right)
\lesssim s^{-\frac \beta 2 - \frac 1{10} } \left( s^{-\frac \beta 2} + s^{-(1-\beta)(\theta-1)}\right) \lesssim s^{-  \beta   - \frac 1{10}}.\nonumber
\end{align}
Using \eqref{dx}, we find
$$
\left| x_s (s) -  \left(\frac 2 {\int Q} c_0 \frac {1}{\theta-1}\right)^2  {x^{-2 \theta+2}(s)} \right| 
\lesssim   s^{-  \beta   - \frac 1{10}} + s^{-\frac 54 -   \beta }
\lesssim s^{-  \beta   - \frac 1{10}}.
$$
and hence\be
\label{vnkrvrnorn}
\left| \left(x^{\frac 1{1-\beta}}\right)_s - (1-\beta)^{-(1-\beta)} \right| \lesssim s^{-\frac 1{10}},
\ee
from $2\theta-1=\frac 1{1-\beta}$ 
and the choice of   $c_0$  in \fref{defcobeta} which gives:
$$
(1-\beta)^{-(1-\beta)}= (2\theta-1)\left(\frac 2 {\int Q} c_0 \frac {1}{\theta-1}\right)^2. 
$$

Since from \fref{initixzero}:
$$
 \frac {(1-\beta) }{s_0^{1-\beta}} x (s_0)=  1,$$
the time integration of \fref{vnkrvrnorn} on $[s_0,s]$ yields
$$
\left|x^{\frac 1{\beta-1}}(s) -  (1-\beta)^{-\frac 1{\beta-1}} s \right| \lesssim s^{1-\frac 1{10}}.
$$
Thus,     
 \be\label{aa2}
\left|  \frac {(1-\beta)}{s^{1-\beta}}  x (s) - 1\right| \lesssim s^{-\frac 1{10}}.
 \ee
 
 Inserting \eqref{aa2} into \eqref{bbb2}, we find for $\l$,
\be\label{aaa2}
\left|  s^{\beta} \l(s) -1\right|  \lesssim s^{-\frac 1{10}}.
\ee
Finally, using $|g(s)| \leq s^{-1 + 2 \beta - \frac 1{5}}$,
we find
\be\label{aaaa2}
\left| \frac {s}{\beta } b(s) -1\right|  \lesssim s^{-\frac 1{10}}.
\ee
From \eqref{aa2}, \eqref{aaa2} and \eqref{aaaa2}, (BS2') follows for $s_0$ large enough.

\subsection{Choice of $\l_0$ and $b_0$ by a topological argument} We now claim from a standard topological argument based on the ougoing behavior of the ODE's for $(f,g)$ that we can find $(b_0,\l_0)\in \mathcal D$ such that the remaining condition $(BS1)$ is closed.\\
Indeed, let $$
G(s) = g(s) s^{1-2 \beta + \frac 1{5}},\quad
F(s) = f(s) s^{\frac \beta 2+ \frac 1{10}},
$$
and
$$
H(s) = F^2(s) + G^2(s).
$$
From (BS2') and (BS3'), since $s^*=s^*(x_0,b_0)<+\infty$, it follows from a standard continuity argument that at $s=s^* \geq s_0,$
\be\label{sature}
H(s^*)=1.
\ee
We first claim the strict outgoing behavior:
\be\label{dFdG}
H'(s^*)\geq \frac {1}{20 s^*}.
\ee
\emph{Proof of \eqref{dFdG}.}
Since
$$
G'(s) = \left(1- 2 \beta +\frac 1{5}\right) g(s)s^{-2 \beta + \frac 1{5}}
+g'(s) s^{1-2 \beta + \frac 1{5}},
 $$
 we have using \fref{dFdG}:
 \be\label{gstar}
 G'(s^*)= \left(1- 2 \beta +\frac 1{5}\right) \frac {G(s^*)}{s^*}
 +O\left((s^*)^{-(1+\frac 1{20})}\right).
 \ee
 Similarly,
 $$
 F'(s^*)= \left(\frac \beta 2+ \frac 1{10}\right) f(s^*) (s^*)^{\frac \beta 2+ \frac 1{10}-1}
 + f'(s^*) (s^*)^{\frac \beta 2+ \frac 1{10}}.
 $$
 We now estimate $f'(s)$. By direct computations and then \eqref{dl}, \eqref{dx}
 and (BS1)-(BS2), 
 \begin{align*}
 f'(s) &= \frac 12 \frac {\l_s}{\l} \l^{\frac 12} - \frac 2{\int Q}c_0 
   x_s x^{-\theta}  \\
 & = \frac 12 \l^{\frac 52} \left[
 \frac b{\l^2}  - \frac 4{\int Q} c_0  \l^{-\frac 32} x^{-\theta}\right]
+O\left(s^{-\frac 54-\frac \beta 2}\right) \\
 &= \frac 12 \l^{\frac 52} g(s) + O\left(s^{-\frac 54-\frac \beta 2}\right) 
 = O\left(s^{-1  -\frac \beta 2 -\frac 1{5}}\right).
 \end{align*}
 Thus, 
 \be\label{fstar}
 F'(s^*)=\left(\frac \beta 2+ \frac 1{10}\right) \frac {F(s^*)} {s^*} + O\left( (s^*)^{-1-\frac 1{10}} \right). 
 \ee
 
 Therefore
 $$
 H'(s^*) = 2 F'(s^*)F(s^*) + 2 G'(s^*) G(s^*) 
 \geq \frac 1{10} \frac {H(s^*)}{s^*} + O\left((s^*)^{-1-\frac 1{20}}\right)
 \geq \frac 1{20s^*},
 $$
for $s_0$ large enough.

\medskip

By  standard arguments (see e.g. the proof of Lemma 6 in \cite{MMC}),
the strict outgoing behavior \fref{dFdG} ensures that the map $(\l_0,b_0)\in \mathcal D \to s^*(x_0,b_0)$
is continuous.
We  define the continuous maps
\begin{align*}
 \l_0(F_0) &= s_0^{-\beta}\left( 1 - F_0 s_0^{-\frac 1{10}} \right)^2
\\
  b_0(F_0,G_0) &= \beta s_0^{-1+\frac \beta 2} \l_0^{\frac 12}(F_0 ) \left( 1+ \frac 1\beta G_0 s_0^{\frac 32\beta - \frac 1{5}}  \l_0^{\frac 32} (F_0)
   \right)\\
  & = \beta s_0^{-1} \left| 1 - F_0 s_0^{-\frac 1{10}}\right|
  \left( 1+ \frac 1\beta G_0 s_0^{- \frac 1{5}}  \left| 1 - F_0 s_0^{-\frac 1{10}}\right|^3
   \right)
\end{align*}
so that
\begin{align*}
G_0 &= s_0^{1-2 \beta +\frac 1{5}} \left( \frac {b_0}{\l_0^2} 
- \frac 4{\int Q} c_0  \l_0^{-\frac 32} x_0^{-\theta}\right)\\
F_0 &= s_0^{\frac \beta 2 +\frac 1{10}} \left( \l_0^{\frac 12} 
+\frac 2{\int Q} c_0 \frac 1{\theta-1} x_0^{-\theta+1} \right).
\end{align*}

Now, consider the continous map
\begin{align*}
\mathcal M \ : \ \   \   \mathcal B_{\RR^2}  &\to  \mathcal S_{\RR^2} ,\\
   (F_0,G_0) &\mapsto \left(F(s^*(\l_0(F_0),b_0(F_0,G_0))), G(s^*(\l_0(F_0),b_0(F_0,G_0)))\right).
\end{align*}
where $\mathcal B_{\RR^2}$ and $\mathcal S_{\RR^2} $ are, respectively, the ball and  the sphere of $\RR^2$ of radius $1$. For $(F_0,G_0)\in \mathcal S_{\RR^2}$, we have
$\mathcal M(F_0,G_0) = (F_0,G_0)$, in other words, 
$\mathcal M$ is the identity on the sphere
$\mathcal S_{\RR^2}$. The existence of such a continuous map 
$\mathcal M$ is  in contradiction with Brouwer's fixed point theorem. Therefore, there exists $\l_0$ and $b_0$ such that
\be\label{estimation}
| \l_0 - s_0^{-\beta}| \leq s_0^{-\beta - \frac 1{10}},
\quad |b_0 - \beta s_0^{-1}| \leq s_0^{-1- \frac 1{10}},
\ee
and $s^*(\l_0,b_0)=+\infty$. 
In particular (BS1)-(BS2)-(BS3) hold on $[s_0,+\infty)$.\\
Finally, \eqref{aa2}, \eqref{aaa2} and \eqref{aaaa2} imply \eqref{sharp}.\\

This concludes the proof of Proposition \ref{PR:BS} and therefore also of Theorem \ref{th:1}.



\appendix
\section{Proof of Lemma \ref{le:qz}}\label{AA}

Recall that  $c_0\in \RR$ and $\theta>1$ are  fixed,  $x_0\gg 1$ is to be taken large enough and $\qz$ is the solution of 
\be\label{defqzA}
\partial_t \qz + \partial_x(\partial_x^2 \qz + \qz^5)=0,\quad
\qz(0,x)= f_0(x),
\ee
where the function $f_0$ is   smooth and satisfies
\be\label{deffzA}
f_0(x)= c_0 x^{-\theta} \hbox{ for $x>\frac {x_0}2$},\quad
f_0(x)=0 \hbox{ for $x<\frac {x_0}4$}, 
\ee
\be\label{estfzA}
\hbox{ for all $x\in \RR$, for all $k\geq 0$},\quad 
\left|\frac {d^{k}f_0}{dx^k} (x)\right| \lesssim c_0 |x|^{-\theta - k}.
\ee

First, for $x_0$ large enough, $\|f_0\|_{L^2}$ is small and it follows from the $L^2$ and $H^s$ Cauchy theory (Corollary 2.9 in \cite{KPV})
that $q_0$ is global and bounded in $H^s$ for all $s\geq 0$, with
$$\sup_t\|q_0(t)\|_{H^s}\lesssim \delta(x_0^{-1}).$$

We define
$$
\qu(t,x)  = \qz(t,x) -f_0(x),
$$
\be\label{defquA}
\partial_t \qu + \partial_x(\partial_x^2 \qu + (\qu+f_0)^5-f_0^5)=F_0,\quad
\qu(0,x)= 0, 
\ee
where
$$
F_0 = -\partial_x^3 f_0 - \partial_x(f_0^5).
$$

For any $\bar \theta\geq 0$, define a smooth function $\varphi_{\bar \theta}$ such that
\begin{equation}\label{dfphiu}
\varphi_{\bar \theta}(x) = x^{\bar \theta}  \text{ for $x\geq 4$},\quad
\varphi_{\bar \theta}(x) =  e^{\frac x8}   \text{ for $x\leq 0$},\quad
\varphi'\geq 0,\  \varphi'''\leq \frac 14 \varphi' \text{ on $\RR$}.
\end{equation}
For 
$$
0\leq \theta_1< 2 \theta+4, \quad \theta_1 \neq 2 \theta+3,
$$
    set
\begin{align*}
M_{\theta_1}(t)& = \int q_1^2(t) \varphi_{\theta_1}\left( x - \frac t 4 - \frac {x_0} 4\right) dx,
\\
E_{\theta_1}(t) &= \int \left( (\partial_x q_1)^2(t) -\frac 13 \left( (q_1+f_0)^6 - f_0^6 - 6 q_1 f_0^5\right)   \right)\varphi_{\theta_1+2}\left( x - \frac t 4 - \frac {x_0} 4 \right) dx,\\\
F_{\theta_1,k}(t) &= \int (\partial_x^k q_1)^2(t) \varphi_{\theta_1+2k}\left( x - \frac t 4 - \frac {x_0} 4\right) dx, \ k\geq 2.
\end{align*}

We differentiate $M_{\theta_1}(t)$ (omitting the variable $x - \frac t 4 - \frac {x_0} 4$ for the function $\varphi_{\theta_1}$):
\begin{align*}
M_{\theta_1}'(t) & = - 3 \int (\partial_x q_1)^2 \varphi_{\theta_1}' -\frac 14 \int q_1^2 \varphi_{\theta_1}'
+ \int q_1^2 \varphi_{\theta_1}'''  \\ &
+ 2 \int \left((q_1+f_0)^5 - f_0^5\right) (q_1 \varphi_{\theta_1})_x + 
2 \left|\int F_0 q_1 \varphi_{\theta_1} \right|\\
& \leq - 3 \int (\partial_x q_1)^2 \varphi_{\theta_1}' -\frac 3{16} \int q_1^2 \varphi_{\theta_1}'
- 2 \int \left[ \frac {(q_1+f_0)^6}6 - \frac   {f_0^6} 6-(q_1+f_0)^5 q_1\right] \varphi_{\theta_1}'
\\ & + 2 \int \left[ (q_1+f_0)^5 - 5 q_1^4 f_0-f_0^5\right] f_0' \varphi_{\theta_1}
+ 2 \left|\int F_0^2  \frac {\varphi_{\theta_1}^2}{\varphi_{\theta_1}'} \right|^{\frac 12}
 \left| \int q_1^2 \varphi_{\theta_1}'\right|^{\frac 12} \\
& \leq - 3 \int (\partial_x q_1)^2 \varphi_{\theta_1}' -\frac 1 8 \int q_1^2 \varphi_{\theta_1}' 
+\delta(x_0^{-1}) \int q_1^2 \varphi_{\theta_1}'   + C \int F_0^2  \frac {\varphi_{\theta_1}^2}{\varphi_{\theta_1}'}     .
\end{align*}
Thus, for $x_0$ large enough, we have obtained
$$
M_{\theta_1}'(t) + \frac 1{10}  \int \left[ (\partial_x q_1)^2 + q_1^2 \right] \varphi_{\theta_1}'
\lesssim \int F_0^2  \frac {\varphi_{\theta_1}^2}{\varphi_{\theta_1}'}   .
$$
For $x_0$ large enough,
\begin{align}
   \int F_0^2  \frac {\varphi_{\theta_1}^2} {\varphi_{\theta_1}'}\left( x - \frac t 4 - \frac {x_0} 4\right)
& \lesssim    \int_{x>\frac t4 + \frac {x_0} 4 + 4} x^{-2(\theta+3)} \left( x - \frac t 4 - \frac {x_0} 4 \right)^{\theta_1+1}  \nonumber   \\
&  + \int_{\frac t8 + \frac {x_0} 8 < x<\frac t4 + \frac {x_0} 4 + 4}   x^{-2(\theta+3)}
+\int_{x< \frac t8 + \frac {x_0} 8} e^{\frac 18 (x-\frac t 4 - \frac {x_0} 4 )}\nonumber
\\ & \lesssim   (t+x_0)^{ \theta_1 -2 \theta-4 } . \label{Fzero}
\end{align}
Therefore, using also $M_{\theta_1}(0)=0$, we find by integration:
\be\label{eqM}
M_{\theta_1}(t) +  \int_0^t\int \left[ (\partial_x q_1)^2 + q_1^2 \right] \varphi_{\theta_1}' \lesssim 
\begin{cases} (t+x_0)^{ \theta_1 -2 \theta-3 } & \hbox{if $0<\theta_1 {-}2 \theta{-}3<1$}  \\
  x_0^{ \theta_1 -2 \theta-3 } & \hbox{if $\theta_1 {-}2 \theta{-}3<0$} \end{cases} \ee

We argue similarly for $E_{\theta_1}(t)$. 
\begin{align*}
E_{\theta_1}'(t) & = 2 \int \partial_t q_1 \left[ - \partial_x^2 q_1 - \left((q_1+f_0)^5 - f_0^5\right)\right] \varphi_{\theta_1+2}\\
&  - 2 \int \partial_t q_1 \partial_x q_1 \varphi_{\theta_1+2}'
- \frac  14 \int\left[ (\partial_x q_1)^2(t) -\frac 13 \left( (q_1+f_0)^6 - f_0^6 - 6 q_1 f_0^5\right)   \right]\varphi_{\theta_1+2}'
\\ &= - \int \left[ \partial_x^2 q_1 + \left((q_1+f_0)^5 - f_0^5\right) \right]^2 \varphi_{\theta_1+2}'\\
& - 2 \int \left[ \partial_x^2 q_1 + (q_1 + f_0)^5 - f_0^5\right] F_0\varphi_{\theta_1+2}\\
& + 2 \int \left[ \partial_x^2 q_1 + (q_1 + f_0)^5 - f_0^5\right]_x \partial_x q_1 \varphi_{\theta_1+2}'
  - 2 \int F_0 \partial_x q_1 \varphi_{\theta_1+2}' 
\\ &- \frac  14 \int\left[ (\partial_x q_1)^2(t) -\frac 13 \left( (q_1+f_0)^6 - f_0^6 - 6 q_1 f_0^5\right)   \right]\varphi_{\theta_1+2}'.
\end{align*}
We use the following computations and estimates
\begin{align*}
&\left| \int \left[ \partial_x^2 q_1 + (q_1 + f_0)^5 - f_0^5\right] F_0\varphi_{\theta_1+2} \right| \\
& \lesssim 
\left|  \int \partial_x q_1 (F_0 \varphi_{\theta_1+2})_x \right| + \left| \int \left[  (q_1 + f_0)^5 - f_0^5\right] F_0\varphi_{\theta_1+2}\right| \\
& \lesssim \frac 1{100} \int   (\partial_x q_1)^2  \varphi_{\theta_1+2}'
+\frac 1{100} \int    q_1^2  \varphi_{\theta_1}'
\\ & + C \int \left(|\partial_x F_0|^2 \frac {\varphi_{\theta_1+2}^2}{\varphi_{\theta_1+2}'}
+ |F_0|^2 \varphi_{\theta_1+2}' + |F_0|^2 \frac {f_0^8 \varphi_{\theta_1+2}^2}{\varphi_{\theta_1}'}\right).
\end{align*}
\begin{align*}
2 \int \partial_x^3 q_1 \partial_x q_1 \varphi_{\theta_1+2}' 
& = - 2 \int (\partial_x^2 q_1)^2 \varphi_{\theta_1+2}' + \int (\partial_x q_1)^2 \varphi_{\theta_1+2}'''\\
& \leq - 2 \int (\partial_x^2 q_1)^2 \varphi_{\theta_1+2}' + \frac 14 \int (\partial_x q_1)^2 \varphi_{\theta_1+2}'.
\end{align*}
\begin{align*}
& \left| 2 \int \left[ (q_1+f_0)^5 - f_0^5\right]_x \partial_x q_1 \varphi_{\theta_1+2}' \right|\\
& =\left| 10 \int (\partial_x q_1)^2 (q_1+f_0)^4 \varphi_{\theta_1+2}'
+10 \int \left( (q_1+f_0)^4 - f_0^4\right) \partial_x f_0  \partial_x q_1 \varphi_{\theta_1+2}'\right|\\
& \leq \frac 1{100} \int (\partial_x q_1)^2 \varphi_{\theta_1+1}' + C \int q_1^2 \varphi_{\theta_1}'.
\end{align*}
\begin{align*}
2 \left| \int F_0 \partial_x q_1 \varphi_{\theta_1+2}'\right| \leq
\frac 1{100} \int (\partial_x q_1)^2 \varphi_{\theta_1+1}' + C \int |F_0|^2\varphi_{\theta_1+2}'.
\end{align*}
Combining these estimates, and using the expression of $F_0$ as in \eqref{Fzero}, we find
\begin{align*}
E_{\theta_1}'(t) \leq - \int (\partial_x^2 q_1)^2 \varphi_{\theta_1+2}' - \frac 1{10} \int (\partial_x q_1)^2 \varphi_{\theta_1+2}' 
+ C\int q_1^2 \varphi_{\theta_1}' + C  (t+x_0)^{ \theta_1 -2 \theta-4 }.
\end{align*}
By integration, and using \eqref{eqM},
\begin{align}
& E_{\theta_1}(t) + \frac 1{10} \int_0^t \int \left[(\partial_x^2 q_1)^2 + (\partial_x q_1)^2 \right]\varphi_{\theta_1+2}'\nonumber \\
& \lesssim \begin{cases} (t+x_0)^{ \theta_1 -2 \theta-3 } & \hbox{if $0<\theta_1 {-}2 \theta{-}3<1$}  \\
  x_0^{ \theta_1 -2 \theta-3 } & \hbox{if $\theta_1 {-}2 \theta{-}3<0$} \end{cases}\label{eqE}
\end{align}
  
We look for an estimate on $\partial_x q_1(t)$ from the above estimate on $E_{\theta_1}(t)$.
Note first that
\begin{align}
\|q_1^2 \sqrt{\varphi_{\theta_1+2}} \|_{L^\infty} &\lesssim
\int_x^{+\infty} |q_1| |\partial_x q_1| \sqrt{\varphi_{\theta_1+2}}
+ \int_x^{+\infty} |q_1|^2 \frac {\varphi_{\theta_1+2}'}{\sqrt{\varphi_{\theta_1+2}}}\nonumber\\
& \lesssim \left( \int q_1^2\right)^{\frac 12} \left( \int |\partial_x q_1|^2  \varphi_{\theta_1+2}
+ \int |q_1|^2 \varphi_{\theta_1}\right)^{\frac 12},\label{Linfty}
\end{align}
so that
\begin{align*}
\int q_1^6 \varphi_{\theta_1+2} \leq \|q_1^2 \sqrt{\varphi_{\theta_1+2}} \|_{L^\infty}^2
\int q_1^2 \lesssim 
\left( \int q_1^2\right)^2 \left( \int |\partial_x q_1|^2  \varphi_{\theta_1+2}
+ M_{\theta_1}(t)\right) .
\end{align*}
Also,
$$
\int q_1^2 f_0^4 \varphi_{\theta_1+2} \lesssim \int q_1^2 \varphi_{\theta_1} = M_{\theta_1}(t).
$$
Thus,
\begin{align}
E_{\theta_1}(t) & \geq \int (\partial_x q_1)^2 \varphi_{\theta_1+2} - C \int (q_1^6 + q_1^2 f_0^4) \varphi_{\theta_1+2}\nonumber\\
& \geq \frac 12 \int (\partial_x q_1)^2 \varphi_{\theta_1+2} - C M_{\theta_1}(t),\label{eq6}
\end{align}
and so
\begin{align}
&\int (\partial_x q_1)^2 \varphi_{\theta_1+2}+ \int_0^t \int \left[(\partial_x^2 q_1)^2 + (\partial_x q_1)^2 \right]\varphi_{\theta_1+2}'\nonumber \\
&\lesssim \begin{cases} (t+x_0)^{ \theta_1 -2 \theta-3 } & \hbox{if $0<\theta_1 {-}2 \theta{-}3<1$}   \\
  x_0^{ \theta_1 -2 \theta-3 } & \hbox{if $\theta_1 {-}2 \theta{-}3<0$} \end{cases} \label{eqEbb}\end{align}
 
Note also that for  any $x$,
\begin{align}
  q_1^2(t,x) \varphi_{\theta_1+1}\left( x - \frac t4-\frac {x_0} 4\right) 
& \lesssim \int_x^{+\infty}  |q_1| |\partial_x q_1|  \varphi_{\theta_1+1} 
+\int_x^{+\infty} |q_1|^2 \varphi_{\theta_1+1}' \nonumber \\
& \lesssim 
\int    |\partial_x q_1|^2  \varphi_{\theta_1+2} 
+\int  |q_1|^2 \varphi_{\theta_1} ,
\label{linfty}\end{align}
and, with $\theta_1 = 2 \theta+ \frac {15}4$,
using the properties of $\varphi_{\theta_1+1}$, for $x>\frac 12 (t+x_0)$, 
\be\label{A11}
|q_1(t,x)| \lesssim x^{- (\frac {\theta_1}2 +\frac 12 )} t^{\frac 38} =x^{- (\theta+\frac {19}8)} t^{\frac 38}
  \lesssim   x^{- (\theta+2)}.
\ee

Finally, we briefly treat the case of higher order derivatives.
We use an induction argument, assuming at the rank $k$  that for all
$1\leq k'<k$,  for all $x$, $t$
\begin{align}\nonumber
& \int_0^t \int (\partial_x^{k'} q_1)^2(t) \varphi_{\theta_1+2k'}' \left( x - \frac t4-\frac {x_0} 4\right) dt
+
\int (\partial_x^{k'} q_1(t))^2 \varphi_{\theta_1+2k'}\left( x - \frac t4-\frac {x_0} 4\right) \\ &\label{ply2}
+|\partial_x^{k'-1} q_1(t)|^2 \varphi_{\theta_1+2k'-1}\left( x - \frac t4-\frac {x_0} 4\right) 
  \lesssim \begin{cases} (t+x_0)^{ \theta_1 -2 \theta-3 } & \hbox{if $0<\theta_1 {-}2 \theta{-}3<1$}  \\
  x_0^{ \theta_1 -2 \theta-3 } & \hbox{if $\theta_1 {-}2 \theta{-}3<0$} \end{cases}
\end{align}
we prove the same estimates for $k'=k$ using $F_{\theta_1,k}$.

Indeed, by simple computation:
\begin{align*}
F_{\theta_1,k}'(t) & = 2 \int (\partial_x^k  q_1)_t (\partial_x^k q_1) \varphi_{\theta_1+2k} -
\frac 12 \int (\partial_x^k q_1)^2 \varphi_{\theta_1+2k}'\\
& = - 2 \int (\partial_x^{k+3} q_1) (\partial_x^k q_1) \varphi_{\theta_1+2k} 
- 2 \int \partial_x^{k+1} \left( (q_1+f_0)^5 -f_0^5\right) \partial_x^k q_1 \varphi_{\theta_1+2k}\\
& - \frac 12 \int (\partial_x^k q_1)^2 \varphi_{\theta_1+2k}'+ \int (\partial_x^k F_0) (\partial_x^k q_1) \varphi_{\theta_1+2k} \\
& \leq - 3 \int (\partial_x^{k+1} q_1)^2 \varphi_{\theta_1+2k}' - \frac 14 \int (\partial_x^k q_1)^2 \varphi_{\theta_1+2k}'\\
& + 2 \int \partial_x^k \left( (q_1+f_0)^5 - f_0^5\right) 
\left( (\partial_x^{k+1} q_1) \varphi_{\theta_1+2k} + (\partial_x^k q_1) \varphi_{\theta_1+2k}'\right)
\\ & + \int (\partial_x^k F_0) (\partial_x^k q_1) \varphi_{\theta_1+2k}.
\end{align*}
We claim, arguing as \eqref{Fzero},
$$
\left| \int (\partial_x^k F_0) (\partial_x^k q_1) \varphi_{\theta_1+2k} \right| 
\leq \frac 1{100}\int (\partial_x^k q_1)^2  \varphi_{\theta_1+2k}'+ C t^{\theta_1-2 \theta -4}.
$$
Next, we claim
\begin{align}
& \left|\int \partial_x^k \left( (q_1+f_0)^5 - f_0^5\right) 
\left( (\partial_x^{k+1} q_1) \varphi_{\theta_1+2k} + (\partial_x^k q_1) \varphi_{\theta_1+2k}'\right)\right|\nonumber \\
& \leq \frac 1{100}\int (\partial_x^{k+1} q_1) \varphi_{\theta_1+2k}'+\frac 1{100}\int (\partial_x^{k} q_1) \varphi_{\theta_1+2k}'+\sum_{k'=0}^{k-1} \int (\partial_x^{k'} q_1)^2  \varphi_{\theta_1+2k'}' .\label{ply3}
\end{align}

Indeed, looking for example at the purely nonlinear term in $q_1$, we have
\begin{align*}
&\left|\int  \partial_x^k (q_1^5)  (\partial_x^{k+1} q_1) \varphi_{\theta_1+2k} \right|
+ \left|\int  \partial_x^k (q_1^5)  (\partial_x^{k} q_1) \varphi_{\theta_1+2k}' \right|
\\
&\leq C \int \left(\partial_x^k (q_1^5)\right)^2 \frac {\varphi_{\theta_1+2k}^2}{\varphi_{\theta_1+2k}'}
+ \frac 1{200}\int (\partial_x^{k} q_1)^2 \varphi_{\theta_1+2k}'+\frac 1{200} \int (\partial_x^{k+1} q_1)^2  \varphi_{\theta_1+2k}' ,
\end{align*}
and
\begin{align*}
& \int \left(\partial_x^k (q_1^5)\right)^2 \frac {\varphi_{\theta_1+2k}^2}{\varphi_{\theta_1+2k}'}
\lesssim
\int \left(\partial_x^k (q_1^5)\right)^2 \varphi_{\theta_1+ 2 k_1+2k_2+2k_3+2k_4+2k_5+1}\\
& \lesssim \sum_{ \genfrac{}{}{0pt}{2}{k_1+k_2+k_3+k_4+k_5=k}{k_5\geq k_4\geq k_3\geq k_2\geq k_1}
} \left( \Pi_{l=1}^3 \| (\partial_x^{k_l} q_1 )^2 {\varphi_{2k_l+2}} \|_{L^\infty} \right) 
\| (\partial_x^{k_4} q_1 )^2 {\varphi_{2k_4-2}} \|_{L^\infty}\int |\partial_x^{k_5} q_1|^2\varphi'_{\theta_1+2k_5-1} 
\\
& \lesssim \delta(x_0^{-1})\sum_{k'=0}^k \int (\partial_x^{k'} q_1)^2  \varphi_{\theta_1+2k'-1}'.  \end{align*}
where the $L^\infty$ norms above are estimated using \eqref{ply2}.
The other terms, all containing $f_0$, are similar and easier.

By integration of $F'_k$ using \eqref{ply3} and \eqref{ply2}, we obtain
$$
F_{\theta_1,k}(t) + \int_0^t \int (\partial_x^{k+1} q_1)^2 \varphi_{\theta_1+2k}' \lesssim  
\begin{cases} (t+x_0)^{ \theta_1 -2 \theta-3 } & \hbox{if $0<\theta_1 {-}2 \theta{-}3<1$}  \\
  x_0^{ \theta_1 -2 \theta-3 } & \hbox{if $\theta_1 {-}2 \theta{-}3<0$} \end{cases} 
$$
Arguing as in the proof of \eqref{linfty}, we prove \eqref{ply2}   for $k'=k$.
The induction argument being complete, we finish the proof as in \eqref{A11}.

\section{Proof of monotonicity results on $\e$}\label{appA}

\subsection{Proof of Lemma \ref{lemmatail}} 
We compute from \fref{eqofeps}:
\begin{align*}
&\frac 12\frac{d}{ds}\int\varphi_{10}\e^2  =   \int \e_s \e \varphi_{10}\\
&=\int\varphi_{10}\e\bigg[\lsl\Lambda \e+\left(-\varepsilon_{yy} + \varepsilon - (\varepsilon +p\Yz +Q_b )^5 + Q_b ^5 +p5 Q^4 \Yz\right)_y      \\
   &      - p_s \Yz + (5Q^4(p-q))_y +\lsl p \Lambda \Yz+\left(\frac {{\lambda}_{s}}{{\lambda}}+{b}\right) {\Lambda} Q_b
+ \left(\frac { x_{{s}}}{\lambda} -1\right) (Q_b  + \varepsilon +p\Yz)_y \\
& + \Phi_{{b}} + \Psi_{{b}} -W_y\bigg].
\end{align*}
We integrate by parts the linear term and use $y\varphi_{10}'={10}\varphi_{10}$ for $y\geq 1$ and $\varphi_{10}'''\ll \varphi_{10}'$ for $y$ large enough to derive the bound
\bee
&&\int\varphi_{10}\e\left[\lsl\Lambda \e+\left(-\varepsilon_{yy} + \varepsilon\right)_y\right] 
\\& = & -\frac12\lsl\int y\varphi_{10}'\e^2-\frac 32\int\varphi_{10}'\e_y^2-\frac12\int\varphi'_{10}\e^2+\frac 12\int\varphi_{10}'''\e^2\\
& \leq & -\frac{10}2\lsl\int \varphi_{10}\e^2-\frac 14\int\varphi_{10}'(\e_y^2+\e^2)+C\mathcal N_{1,\rm loc}.
\eee

 By integration by parts in the nonlinear term, we can remove all derivatives on $\e$ to obtain
 (using $|Q_b|+|(Q_b)_y|\leq Ce^{-\frac 12 {y}}$ for $y>0$) 
\bee
 \left|\int\varphi_{10}\e \left[(\varepsilon +Q_b )^5 -Q_b ^5\right]_y\right| 
&\lesssim & \int_{y>0} \varphi_{10}e^{-\frac 12 {y}}\e^2(|\e|^3 +1) +\int\varphi_{10}'\e^6\\
&  \lesssim &  \int_{y>0} e^{-\frac 14 {y}} \e^2(|\e|^3 +1)+\int\varphi_{10}'\e^6\eee
Thus, by standard Sobolev estimates,
\bee
 \left|\int\varphi_{10}\e \left[(\varepsilon +Q_b )^5 -Q_b ^5\right]_y\right|  
 \lesssim \mathcal N_{1,\rm loc} +\delta(\alpha^*)\int\varphi'_{10}(\e_y^2+\e^2).
\eee

Next, by the bootstrap estimates,
\begin{align*}
& \left| \int \varphi_{10} \e \left[ (\e + Q_b +p\Yz)^5 - 5 p Q^4 \Yz - (\e+Q_b)^5\right]_y\right| \\
& \lesssim \int_{y>10}   \e   e^{-\frac 12 y}\left(p^2 + |b||p| + |p| (|\e|+|\e|^5 \right)
\lesssim \mathcal N_{1,\rm loc} +\frac 1{s^2}.
\end{align*}

By \eqref{BSbsps} and $\Yz \in \mathcal{Y}$,
$$
\left| p_s  \int \Yz \varphi_{10} \e\right| \lesssim 
\mathcal N_{1,\rm loc} +\frac 1{s^2}.
$$

By \eqref{cl.2} and \eqref{BSp},
\begin{align*}
& \left| \int 5 (Q^4 (p-q))_y  \e \varphi_{10}\right|
=\left| \int 5 (p-q) e^{-4 |y|} (\e \varphi_{10})_y \right|
\\ &\leq C \frac 1{s^2}  \int (|\e_y|+|\e|) e^{-|y|}\varphi_{10}
\leq \frac 1{100} \int \e_y^2  \varphi'_{10} + C \mathcal N_{1,\rm loc} +\frac C{s^2}.
\end{align*}

By \eqref{BSparam},
$$
\left|\lsl p  \int \Lambda \Yz \varphi_{10} \e\right|
\lesssim \mathcal N_{1,\rm loc} +\frac 1{s^2}.
$$

The terms involving the geometrical parameters are controlled from the exponential localization of $Q_b$ on the right and \fref{BSlsl}--\fref{BSparam}:
$$\left|\lsl+b\right|\left|\int\varphi_{10}\e(\Lambda Q_b)\right|\lesssim \left(\frac 1{s^2} +\mathcal N_{1,\rm loc}^{\frac12}\right)\mathcal N_{i,\rm loc}^{\frac12}\lesssim \mathcal N_{1,\rm loc}+\frac 1{s^2},$$
\begin{align*}\left|\xsl-1\right|\left|\int\varphi_{10}\e(Q_b+\e+p\Yz)_y\right|&\lesssim \left(\frac 1{s^2}+\mathcal N_{1,\rm loc}^{\frac12}\right)\left[\mathcal N_{1,\rm loc}^{\frac12}+\int\varphi_{10}'\e^2\right]\\ &\lesssim \mathcal N_{1,\rm loc}+\frac 1{s^2}+\delta(\alpha^*) \int\varphi_{10}'\e^2,\end{align*}
$$\int\left|\varphi_{10}\e\Phi_b\right|\lesssim |b_s| \mathcal N_{1,\rm loc}^{\frac 12}\lesssim \frac 1{s^2} +\mathcal N_{1,\rm loc}.$$
We control similarily the interaction with the error from \fref{eq:202}: $$\int\left|\varphi_{10}\e\Psi_b\right|\lesssim \frac 1{s^2} \mathcal N_{1,\rm loc}^{\frac 12}\lesssim \frac 1{s^2}+\mathcal N_{1,\rm loc}.$$

Finally, we claim
\be\label{Wdix}
\left| \int \varphi_{10} \e W_y\right|
\leq C \mathcal N_{1,\rm loc}+\frac C{s^2}  +  \frac 1{50} \int \e^2 \varphi_{10}'
\ee
We only treat the first term in $W$, the other terms are similar and easier.
First,  integrating by parts, we remove the derivative from $\e$ to obtain derivative on $q$, $\varphi_{10}$.
Indeed,
\begin{align*}
& - \int \left[ 5 (Q_b+\e+p\Yz )^4 q - 5 Q^4 q\right]_y \e \varphi_{10}  = \int \left[ 5 (Q_b+\e+p\Yz )^4 q - 5 Q^4 q\right]  (\e \varphi_{10})_y\\
& =\int \bigg\{\left[\left(Q_b+\e+p\Yz\right)^5- (Q_b+p\Yz)^5 - 5Q^4 \e\right]_y
\\
&- 5(Q_b+p\Yz)_y\left[(Q_b+\e+p\Yz)^4-(Q_b+p\Yz)^4\right] + 20 Q^3 Q_y \e \bigg\}q\varphi_{10}  
\\ &+ 5\int \left[   (Q_b+\e+p\Yz )^4   -   Q^4 \right]   \e q\varphi_{10}' 
\\
& = - \int \left[\left(Q_b+\e+p\Yz\right)^5-(Q_b+p\Yz)^5 - 5Q^4 \e\right] (q_y\varphi_{10}+q \varphi_{10}') 
\\
&- 5\int \left\{(Q_b+p\Yz)_y \left[(Q_b+\e+p\Yz)^4-(Q_b+p\Yz)^4\right] - 4 Q^3 Q_y \e \right\}  q\varphi_{10} 
\\ & + 5 \int \left[   (Q_b+\e+p\Yz )^4  -   Q^4  \right]   \e q \varphi_{10}'
.
\end{align*}
From the above expression, we obtain for $s$ large enough (using $\|\e\|_{L^\infty} \lesssim \|\e\|_{H^1}\lesssim 
\delta(\alpha^*)$ and $\|q\|_{L^\infty(y>0)} \lesssim \frac 1s$)
\begin{align*}
&\left| \int \left[ 5 (Q_b+\e+p\Yz )^4 q - 5 Q^4 q\right]_y \e \varphi_{10} \right|\\
& \lesssim 
\int (|q|+|q_y|) e^{-\frac {|y|}2} (|\e|+|p|+|b|)|\e|
+\int (|q_y|\varphi_{10} + |q|\varphi_{10}') |\e|^5 \\
& \leq  
C \mathcal N_{1,\rm loc}+\frac C{s^2}  +  \frac 1{100} \int \e^2 \varphi_{10}'+\delta(\alpha^*)\int   |\e|^2|q_y|\varphi_{10} .
\end{align*}
To control the last term above, we use
$q_y(t,y)=\l^{\frac 32} \partial_x q_0(t,\l y + x(t)) $ so that 
$
|q_y(t,y)|\lesssim \l^{\frac 32}  (\l y +x(t))^{-\theta-1} 
\lesssim \l^{\frac 12} x^{-\theta}(t) (y+1)^{-1},
$ for $y>10$,
and
$$
\int |\e|^2|q_y|\varphi_{10} \lesssim
\frac {\l^{\frac 12}}{x^{\theta}} \int_{y>0} \e^2 (y+1)^{-1} \varphi_{10}
\leq  \frac 1{100} \int \e^2 \varphi_{10}'.
$$
The collection of above estimates yields the bound: 
$$\frac{d}{ds}\int\varphi_{10}\e^2+{10}\lsl\int\varphi_{10}\e^2\lesssim \mathcal N_{1,\rm loc}+\frac 1{s^2},$$ and \fref{keyestimate} is proved.

\subsection{Proof of Lemma  \ref{propasymtp}}

 {\bf step 1} Weighted $L^2$ controls at the right.
 
We first recall from \cite{MMR1}, proof of Proposition 3.1, the following controls for all $s\in [0,s_0]$,
\be
\label{weightedone}
\int_{y>0}y\e^2(s)\lesssim \left(1+\frac{1}{\l^{\frac{10}9}(s)}\right)\mathcal{N}_{1,\rm loc}^{\frac 89}(s),
\ee
\be
\label{weightedonebis}
\int_{y>0}y^2\e^2(s)\lesssim \left(1+\frac{1}{\l^{\frac {10}9}(s)}\right)\mathcal{N}_{2,\rm loc}^{\frac 89}(s),
\ee
\be
\label{lonebound}
\int_{y>0}|\e(s)|\lesssim \mathcal N_2^{\frac{1}{2}}(s).
\ee
 \medskip

{\bf step 2} Algebraic computations on $\mathcal F_{i}$.
 
First, note that the equation of $\e$ \eqref{eqofeps} can be rewritten as follows:
\begin{align}  \nonumber  
  \partial_s \varepsilon -\lsl  \Lambda \e & = \left(-\partial_{y}^2 \e  + \e  - \zz\right)_y
\\ & -p_s \Yz +(  5 Q^4 (p-q) )_y +\lsl p\Lambda \Yz
 + \left(\frac {{\lambda}_{s}}{{\lambda}}+{b}\right) {\Lambda} Q_b 
\nonumber \\
&+ \left(\frac { x_{s}}{\lambda} -1\right) (Q_b + \varepsilon+p  \Yz)_y   
+ \Phi_{b}  
 + \Psi_{{b}} ,
\label{eqebis}
\end{align}
where
$$
\zz=(Q_b+p\Yz+q+ \varepsilon )^5 -Q_b^5 - 5Q^4 (p\Yz+q)-q^5,
$$
$$\Phi_b = - {b}_{s} \left(\chi_b   + \gamma   y (\chi_b)_y\right) P,\quad  -\Psi_b=\left(Q_b''- Q_b+ Q_b^5\right)'+b {\Lambda} Q_b.$$

We compute
\begin{align*}
  s^{-j} \frac d{d{s}}\left[ s^j {\cal F_{i}} \right]  &=
 2\int \psi_B(\varepsilon_y)_s \varepsilon_y\\
 & + 2\int \varepsilon_s \left[ \varepsilon \varphi_{i,B} - \psi_B Z\right]\\
&  -2 \int \psi_B(Q_b)_s \left[(\varepsilon+Q_b+p\Yz+q)^5 - (Q_b+p\Yz+q)^5 - 5 \varepsilon Q_b^4\right]\\
& - 2 \int \psi_B (p_s \Yz+ q_s) \left[(\varepsilon+Q_b+p\Yz+q)^5 -(Q_b+p\Yz+q)^5  - 5 \varepsilon Q^4\right]
\\& + 10 \int \psi_B q_s q^4 \e + \frac j s \mathcal F_{i}
\end{align*}
which we rewrite 
\be
\label{eqfnoneoge}
 s^{-j} \frac d{d{s}}\left[ s^j {\cal F_{i}} \right]  =  f^{(i)}_{1}+f^{(i)}_2+f^{(i)}_3,
 \ee
  where 
{\allowdisplaybreaks
\begin{align*}
  f^{(i)}_{1} &=
      2 \int \left( \varepsilon_{s} - {\frac{{\lambda}_s}{{\lambda}}} {\Lambda} \varepsilon\right) \left( - (\psi_B\e_y)_{y} + \varepsilon \varphi_{i,B} -\psi_BZ\right) ,\\
f^{(i,j)}_2  &= 
 2 {\frac{{\lambda}_s}{{\lambda}}}  \int    {\Lambda} \varepsilon  \left( - (\psi_B\varepsilon_y)_y +  \varepsilon\varphi_{i,B} -\psi_B Z \right) +\frac js\mathcal F_{i}, \\
 f^{(i)}_3 &= 
  - 2  \int \psi_B(Q_b)_{s} \left[(\varepsilon+Q_b+p\Yz+q)^5 - (Q_b+p\Yz+q)^5 - 5 \varepsilon Q_b^4\right]\\
& - 2 \int \psi_B (p_s \Yz+ q_s) \left[(\varepsilon+Q_b+p\Yz+q)^5  -(Q_b+p\Yz+q)^5 - 5 \varepsilon Q^4\right]\\
&+ 10 \int \psi_B q_s q^4 \e.
\end{align*}
}

We claim the following estimates on the above terms: for some $\mu_0>0$,  
\begin{align}
\label{re1}
& \frac{d }{ds} f^{(i)}_{1}\leq - \mu_0  \int   \left(\varepsilon_y^2+\e^2\right)   \varphi'_{i,B}  + 
 C s^{-4},
\\  
\label{re2}
& \left|\frac{d }{ds} f^{(i)}_{k}\right|\leq \frac {\mu_0}{10}    \int   \left(\varepsilon_y^2+\e^2\right)   \varphi'_{i,B}  +  C s^{-4}+  C s^{10\beta -9}, \quad \hbox{for $k=2,3.$}
\end{align}
  Inserting \eqref{re1} and \eqref{re2} into \fref{eqfnoneoge}   yields \fref{lyapounovconrolbis}
  for all $j$.
In steps 3 - step 5, we prove \eqref{re1} and \eqref{re2}.

Observe that the  definitions of $\varphi_i$ and $\psi$ imply the following estimates:
\begin{align} \label{defphi4}
    \forall y\in \RR, \quad  &|\varphi_i'''(y) | + | \varphi_i''(y) |+|\psi'''(y)|+|y\psi'(y)|+|\psi(y)|
   \lesssim  \varphi'_i(y)\lesssim \varphi_i(y),\\
   \forall y\in (-\infty,2], \quad 
  \label{defphi4bis}
  &    e^{ |y|} \psi(y)  + e^{  |y|} \psi'(y) +\varphi_i(y)\lesssim \varphi_i'(y),\\
  \forall y\in \RR, \quad  &
  \varphi_2'(y)\lesssim \varphi_1(y) \lesssim \varphi'_2(y). \label{defphi4tri}
\end{align}
In particular,
\be\label{locpasloc}
{\mathcal{N}_{1,\rm loc}} (s) \lesssim{\mathcal{N}_{2,\rm loc}} (s) \lesssim  {\mathcal{N}_{1}} (s)\lesssim {\mathcal{N}_2} (s),\quad
\int\varepsilon^2 (s,y)     \varphi_{1,B} (y) dy\lesssim {\mathcal{N}_{2,\rm loc}} (s) .
\ee

 {\bf step 3} Control of $f_1^{(i)}$. Proof of \eqref{re1}. 
 We compute 
$ f_{1}^{(i)} $
using \eqref{eqebis}
\begin{align*}
  \rm f_{1}^{(i)}& = 2 \int   \left(  -\varepsilon_{yy} {+} \varepsilon  {-}\zz  \right)_y \left(-(\psi_B\e_y)_y{+}\e\varphi_{i,B}{-}\psi_B\zz \right) 
\\ & +2\left(\frac {{\lambda}_{s}}{{\lambda}}+{b}\right) \int {\Lambda} Q_b  \left( - (\psi_B\varepsilon_y)_y + \varepsilon \varphi_{i,B}-\psi_B \zz   \right)
\\ &   + 2 \left(\frac { x_{{s}}}{\lambda} -1\right) \int (Q_b  + \varepsilon+p\Yz)_y
 \left( -(\psi_B \varepsilon_y)_y + \varepsilon \varphi_{i,B}-\psi_B\zz  \right)
 \\ & + 2 \int  \Phi_{{b}}  \left( -(\psi_B \varepsilon_y)_y + \varepsilon \varphi_{i,B}-\psi_B\zz   \right)
\\ & + 2 \int  \Psi_{{b}}   \left( -(\psi_B \varepsilon_y)_y + \varepsilon \varphi_{i,B}-\psi_B\zz  \right) \\
& + \int \left(-p_s\Yz +(5Q^4(p-q))_y + \lsl p \Lambda \Yz\right) \left( -(\psi_B \varepsilon_y)_y + \varepsilon \varphi_{i,B}-\psi_B\zz  \right) 
\\ & = {{\rm f}^{(i)}_{1,1}}+ {{\rm f}^{(i)}_{1,2}}+{{\rm f}^{(i)}_{1,3}}+{{\rm f}^{(i)}_{1,4}}+{{\rm f}^{(i)}_{1,5}}+{{\rm f}^{(i)}_{1,6}}.
\end{align*}
\underline{{\em Term $f^{(i)}_{1,1}$}}:  We first integrate by parts  
\bee
f^{(i)}_{1,1}& = &  2 \int\left[  -\varepsilon_{yy} + \varepsilon  -\zz  \right]_y\left[-\e_{yy}+ \varepsilon  -\zz\right] \psi_B \\
& + & 2\int\left[  -\varepsilon_{yy} + \varepsilon  -\zz\right]_y\left(-\psi_B'\e_y+\e(\varphi_{i,B}-\psi_B)\right). 
\eee
We compute the various terms  :
\begin{align*} & 2\int\left[  -\varepsilon_{yy} + \varepsilon  -\zz  \right]_y\left[-\e_{yy}+ \varepsilon  -\zz\right]\psi_B =-\int\psi_B'\left[-\e_{yy}+ \varepsilon  -\zz\right]^2\\
&  =-\int\psi_B'\left[-\e_{yy}+ \varepsilon \right]^2  - \int\psi_B'\left\{\left[-\e_{yy}+ \varepsilon  -\zz\right]^2-\left[-\e_{yy}+\e\right]^2\right\}\\
&   = -\left[\int\psi_B'(\e_{yy}^2+2\e_y^2)+\int\e^2(\psi_B'-\psi_B''')\right]\\
&  -\int\psi_B'\left\{\left[-\e_{yy}+ \varepsilon  -\zz\right]^2-\left[-\e_{yy}+\e\right]^2\right\}.
\end{align*}
Next after integration by parts:
\bee
& & 2 \int\left[  -\varepsilon_{yy} + \varepsilon\right]_y\left[-\psi_B'\e_y+\e(\varphi_{i,B}-\psi_B)\right]\\
&  &  =-2 \Big\{\int\psi_B'\e_{yy}^2+\int\e_y^2(\tfrac32\varphi_{i,B}'-\tfrac12\psi'_B-\tfrac12\psi_B''')\\&&+\int\e^2(\tfrac12(\varphi_{i,B}-\psi_B)'-\tfrac12(\varphi_{i,B}-\psi_B)''')\Big\},
\eee

\bee
&-& 2 \int \zz_y(\varphi_{i,B}-\psi_B)\e =
2 \int \zz (\varphi_{i,B}'-\psi_B')\e + 2\int \zz (\varphi_{i,B}-\psi_B)\e_y\\
& = & 2 \int \zz (\varphi_{i,B}'-\psi_B')\e - \frac 13\int(\varphi_{i,B}-\psi_B)'\\ &&\left\{[(Q_b+\e+p\Yz+q)^6-(Q_b+p\Yz+q)^6-6Q_b^5\e -6q^5 \e -30 Q^4(p\Yz+q)\e  \right\}\\
& - & 2\int (\varphi_{i,B}-\psi_B)(Q_b)_y[(Q_b+\e+p\Yz+q)^5-(Q_b +p\Yz+q)^5-5Q_b^4\e]\\
& + & 40 \int (\varphi_{i,B}-\psi_B) Q' Q^3 (p\Yz+q) \e\\
& - & 2 \int (\varphi_{i,B}-\psi_B)(p \Yz'+q_y) [(Q_b+\e+p\Yz+q)^5-(Q_b +p\Yz+q)^5-5Q^4\e].
\eee

We collect the above computations and obtain the following
\begin{align*}
& f^{(i)}_{1,1} = -\int\left[3\psi_B'\e_{yy}^2+(3\varphi_{i,B}'+\psi_B'-\psi_B''')\e_y^2+(\varphi_{i,B}'-\varphi_{i,B}''')\e^2\right]\\
& -\frac 13 \int \left[ (\varepsilon {+} Q_b {+}p\Yz{+}q )^6{-}  {(Q_b{+}p\Yz{+}q) ^6}  {-}6 Q_b ^5 \varepsilon {-} 6 q^5 \e{-} 30 Q^4(p\Yz{+}q)\e{-}6Z \varepsilon \right](\varphi_{i,B}'{-}\psi_B')\\
& -   2\int (\varphi_{i,B}-\psi_B)(Q_b)_y[(Q_b+\e+p\Yz+q)^5-(Q_b +p\Yz+q)^5-5Q_b^4\e]\\
& + 40 \int (\varphi_{i,B}-\psi_B) Q' Q^3 (p\Yz+q) \e\\
& -   2 \int (\varphi_{i,B}-\psi_B)(p  \Yz'+q_y) [(Q_b+\e+p\Yz+q)^5-(Q_b +p\Yz+q)^5-5Q^4\e]\\
& +    2 \int Z_y \e_y \psi_B'  -  \int\psi_B'\left\{\left[-\e_{yy}+ \varepsilon  -\zz \right]^2-\left[-\e_{yy}+\e\right]^2\right\}\\
& =  (f^{(i)}_{1,1})^<+(f^{(i)}_{1,1})^{\sim}+(f^{(i)}_{1,1})^>
\end{align*}
where $(f^{(i)}_{1,1})^{<,\sim,>}$ respectively corresponds to integration on $y<-\frac B2$, $|y|\leq \frac B2$, $y>\frac B2$.

We recall
\be
\label{nonlinearsobolev}
 \|\e \|_{L^\infty} \lesssim \|\e\|_{H^1} \lesssim  {\delta(\alpha^*)}  .
\ee

$\bullet$ For the region $y<-B/2$, we rely on monotonicity type arguments and estimate using \eqref{defphi4}:
$$
\int_{y<-B/2} \varepsilon^2 |\varphi'''_{i,B}|\lesssim
  \frac 1 {B^2}  \int_{y<-B/2} \varepsilon^2 \varphi'_{i,B} \leq  \frac 1{100} \int_{y<-B/2} \varepsilon^2 \varphi'_{i,B},
$$
$$\int_{y<-B/2} \varepsilon_y^2|\psi'''_B|\lesssim
  \frac 1 {B^2}  \int_{y<-B/2} \varepsilon_y^2 \varphi'_{i,B} \leq  \frac 1{100} \int_{y<-B/2} \varepsilon_y^2 \varphi'_{i,B},$$
by choosing $B$ large enough.  
By \eqref{nonlinearsobolev} (for $B$ large and $\alpha^*$ small)
\begin{align*}
& \left| \int_{y<-B/2}  \left[ (\varepsilon {+} Q_b {+}p\Yz{+}q )^6{-}  {(Q_b{+}p\Yz{+}q) ^6}  {-}6 Q_b ^5 \varepsilon {-} 6 q^5 \e{-}30 Q^4(p\Yz{+}q)\e{-} 6Z \varepsilon \right](\varphi_{i,B}'{-}\psi_B') \right|
\\ 
& \lesssim    
 \int_{y<-B/2} \left(\delta(\alpha^*)  + (|Q_b| ^4+|p|^4+|q|^4) \varepsilon^2
 +(|b|+|p|+|q|)^2 |\e| \right)\varphi'_{i,B} 
 \\& \lesssim  \left(\delta(\alpha^*)+\delta(s_0^{-1})+e^{-\frac B{10} }\right)\int_{y<-B/2}  \varphi_{i,B}' \e^2 
 +\frac 1{s^2} \left( \int_{y<-B/2}  \varepsilon^2   \varphi'_{i,B}\right)^{\frac 12}
 \\
& \leq    \frac 1{100} \int_{y<-B/2}  \varepsilon^2   \varphi'_{i,B}+\frac 1{s^4},
\end{align*}
where we have used from the definition of $q$ and \eqref{BSp}
$$
\left(\int q^4  \varphi'_{i,B}\right)^{\frac 12} \lesssim \frac 1{s^2}.
$$
Similarily for $\alpha^*$ small depending on $B$,
\bee
&& \bigg| \int_{y<-\frac{B}{2}}  (\varphi_{i,B}-\psi_B)(Q_b)_y\left[(Q_b+\e+p\Yz+q)^5-(Q_b +p\Yz+q)^5-5Q_b^4\e\right] \\ &&-   20 \int (\varphi_{i,B}-\psi_B) Q' Q^3 (p\Yz+q) \e 
\bigg|\\
&& \lesssim   B \int_{y<-\frac B2}\left(|\e|^5+\e^2(|Q_b|^3+|p|^3+|q|^3)\right)(|Q_y|+|b||(P\chi_b)'|)\varphi_{i,B}' \\
&&+B \int |\e|\left(  |p|^2+|q|^2\right)
(|Q_y|+|b||(P\chi_b)'| )\varphi_{i,B}'\\
&&+B \int |\e|  (|p|+|q|)   
\left| (Q_b)_y Q_b^3 - Q'Q^3\right|\varphi_{i,B}'\\&&\leq \frac 1{100} \int_{y<-B/2}  (\e_y^2+\varepsilon^2)  \varphi'_{i,B}+\frac 1{s^4}.
\eee
The next term in $(f^{(i)}_{1,1})^<$ is
$$- 2\int (\varphi_{i,B}-\psi_B)(p \Yz'+q_y) [(Q_b+\e+p\Yz+q)^5-(Q_b +p\Yz+q)^5-5Q^4\e].$$
To estimate it, we note the following
\begin{align*}
&\left|(Q_b+\e+p\Yz+q)^5-(Q_b +p\Yz+q)^5-5Q^4\e \right|\\
& \lesssim   |\e|^5 +|\e|^2+|\e| (|b|+|p|+|q|).
\end{align*}
Now, using
\be\label{BSqy}
\int  |p|  \varphi_{i,B} dy \lesssim \frac 1{s}, 
\quad \int  |q_y|    \varphi_{i,B} dy \lesssim \frac 1{s^2},  \ee
we obtain proceeding as before
\begin{align*}
&\left|\int (\varphi_{i,B}-\psi_B)(p  \Yz'+q_y) [(Q_b+\e+p\Yz+q)^5-(Q_b +p\Yz+q)^5-5Q^4\e]
\right|\\ &\lesssim     \frac 1{100} \int_{y<-B/2} (\e_y^2+\varepsilon^2)  \varphi'_{i,B}+\frac 1{s^4}.
\end{align*}
  We further estimate using \fref{nonlinearsobolev} and   $(\varphi'_i)^2\lesssim \psi'\lesssim (\varphi'_i)^2$ for $y<-\frac 12$:
\bee
&& \left|\int Z_y \e_y \psi_b'\right| \\&& \lesssim 
\left|\int_{y<-\frac B2}\psi_B'\e_y\left\{(Q_b)_y[(Q_b+\e+p\Yz+q)^4-Q_b^4]-Q'Q^3(p\Yz+q)\right\} \right|\\
&&+\left|\int_{y<-\frac B2}\psi_B'\e_y   (p\Yz'+q_y) ((Q_b+\e+p\Yz+q)^4-Q^4) \right| + \left| \int q_y q^4 \e\psi_B'\right|\\
&& +\left| \int_{y<-\frac B2}\psi_B'\e_y^2 (Q_b+\e+p\Yz+q)^4\right|\\
&& \lesssim  \left( e^{- \frac 12 B} +\delta(s_0^{-1}) +\delta(\alpha^*)\right) \int_{y<-\frac B2}\varphi'_{i,B}(\e_y^2+\e^2) 
+\frac 1{s^4}\\
&& \leq 
 \frac 1{100}  \int_{y<-B/2}  (\e_y^2+\varepsilon^2)  \varphi'_{i,B}+\frac C{s^4}.
\eee
 
The remaining nonlinear term is estimated using the local $H^2$ control provided by localization
(see more details in \cite{MMR1})
\bee
& & \left|\int_{y<-\frac B2}\psi_B'\left\{\left[-\e_{yy}+ \varepsilon  -Z\right]^2-\left[-\e_{yy}+\e\right]^2\right\}\right|
\\
&=&   \left|\int_{y<-\frac B2}\psi_B' \left(-2\e_{yy}+ 2\varepsilon  -Z\right)  Z  \right|\\
& \leq & \frac{1}{100}\int_{y<-\frac B2}\psi_{B}'(|\e_{yy}|^2+|\e|^2)+C\int_{y<-\frac B2}(\varphi'_{i,B})^2Z^2\\
& \leq &  \frac 1{100} \int_{y<-B/2} \left[ \e_{yy}^2\psi_B'+(\e_y^2+\varepsilon^2)  \varphi'_{i,B}\right]
+\frac 1{s^4}.
\eee

$\bullet$ In the region $y>\frac B2$, we have $\psi_B(y)=1$, so that several terms cancel in $f_{1,1}^{(i)}$. For the remainding terms, we argue as before.
We rely on \eqref{defphi4} to estimate:
$$\int_{y>B/2} \varepsilon^2 |\varphi'''_{i,B}|\lesssim
  \frac 1 {B^2}  \int_{y>B/2} \varepsilon^2 \varphi'_{i,B} \leq  \frac 1{100} \int_{y>B/2} \varepsilon^2 \varphi'_{i,B},$$ and we use the exponential localization of $Q_b$ to the right and \eqref{nonlinearsobolev},   to control:  
  \begin{align*}
& \left| \int_{y>B/2} \left(\frac {(\varepsilon {+} Q_b {+}p\Yz{+}q )^6} 6
-\frac {(Q_b{+}p\Yz{+}q)^6} 6 -Q_b^5 \e - 5 Q^4 (p\Yz{+}q) - Z \e  \right)  \varphi_{i,B}' \right|
\\ 
&   \lesssim 
 \int_{y>B/2} \left(\varepsilon^6  + \left(|Q_b| ^4 +p^4 +q^4 \right)\varepsilon^2
 + (|b|+|q|+|p|)^2 |\e|\right)\varphi_{i,B}' \\
 & \lesssim  (\delta(\alpha^*)+\delta(s_0^{-1})+e^{-\frac B{10}})\int_{y>B/2}  \e^2   \varphi_{i,B}' + \frac 1{s^4}\\
& \leq  \frac 1{100} \int_{y>B/2}  \varepsilon^2   \varphi'_{i,B}
+\frac C{s^4},
\end{align*}
\begin{align*}
& \Bigg| \int_{y>B/2} \left[(\varepsilon+Q_b+p\Yz+q )^5 -(Q_b+p\Yz+q) ^5- 5 Q_b ^4 \varepsilon\right](Q_b )_y (\psi_B-\varphi_{i,B}) \\
&-20 \int_{y>B/2}  Q'Q^3 (p\Yz+q) \e (\psi_B-\varphi_{i,B})\Bigg|\\
&  \leq \frac 1{100} \int_{y>B/2}   \varepsilon^2   \varphi'_{i,B} +\frac C{s^4}.
\end{align*}
Since $\Yz\in \mathcal{Y}$, we   argue similarly to obtain
\begin{align*}
&\left| \int (\psi_B-\varphi_{i,B}) p\Yz' 
\left[ (Q_b+\e+p\Yz+q)^5 - (Q_b+p\Yz+q)^5 - 5Q^4 \e\right]\right| \\
&\leq \frac 1{100} \int_{y>B/2}   \varepsilon^2   \varphi'_{i,B} +\frac C{s^4}.
\end{align*}
Next, we have from \eqref{BSqy},
\begin{align*}
&\left| \int (\varphi_{i,B} - \psi_B) q_y \left[ (Q_b+\e+p\Yz+q)^5 - (Q_b+p\Yz+q)^5 - 5Q^4 \e\right]\right| \\
& 
 \leq \frac 1{100} \int_{y>B/2}   \varepsilon^2   \varphi'_{i,B} +\frac C{s^4}.
\end{align*}
Also,
\begin{align*}
\left| \int Z_y \e_y \psi_B'\right| =
\left| \int Z (\e_{yy} \psi_B' + \e_y \psi_B'')\right|
\leq \frac 1{100} \int (\e_{yy}^2 + \e_y^2 + \e^2) \psi_B' +\frac C{s^4}.
\end{align*}

$\bullet$ In the   region $|y|<B/2$, $\varphi_{i,B}(s,y) = 1+ y/B$ and $\psi_B(y)=1$. In particular, $\varphi_{i,B}'''=\psi_B'=0$ in this region, and we obtain:
\begin{align*}
(f^{(i)}_{1,1})^{\sim}  
&    = - \frac {1}{B} \int_{|y|<B/2}  \Big\{ 3 \varepsilon_y^2 + \varepsilon^2  \\
  &   + \frac 13  \left( (\varepsilon + Q_b +p\Yz+q)^6 
-(Q_b+p\Yz+q)^6  - 6 Q_b^5 \varepsilon -30 Q^4(p\Yz+q)\e
 -6 Z\varepsilon \right) 
  \\
  &  + 2\left( (\varepsilon+Q_b+p\Yz+q )^5 -(Q_b+p\Yz+q)^5- 5 Q_b ^4 \varepsilon\right) y (Q_b )_y  
\\
& -40 yQ'Q^3 (p\Yz+q) \e \\
& + 2 y(p \Yz'+q_y)\left((\varepsilon+Q_b+p\Yz+q )^5 -(Q_b+p\Yz+q) ^5- 5 Q_b ^4 \varepsilon\right)\Big\} \\
&    = - \frac {1}   { B} \int_{|y|<B/2}   \left\{ 3 \varepsilon_y^2 + \varepsilon^2 - 5 Q^4 \varepsilon^2 + 20 y Q' Q^3 \varepsilon^2\right\} + R_{\rm Vir}(\varepsilon),
\end{align*}
where
\begin{align*}
R_{\rm Vir}(\varepsilon) = - \frac {1}   { B} & \int_{|y|<B/2} \Bigg\{
\frac 13 \bigg( (\e+Q_b+p\Yz + q)^6 - (Q_b+p\Yz+q)^6 - 6Q_b^5 \e\\ &
-30 Q^4 (p\Yz+q)\e - 6Z\e  - 15 Q^4 \e^2 \bigg)\\
& +2 \left( (\e+Q_b+ p\Yz + q)^5 - (Q_b+p\Yz + q)^5 - 5 Q_b^4 \e\right) y (Q_b)_y
\\ &- 20 y Q' Q^3 (p\Yz+q) \e - 10 y Q' Q^3 \e^2 \\
& +2 y (p\Yz'+q_y) \left( (\e +Q_b+p\Yz+q)^5 - (Q_b+p\Yz +q)^5 - 5Q_b^4 \e\right)\Bigg\} 
\end{align*}
As before, we estimate
\begin{align*}
|R_{\rm Vir}(\varepsilon)|&\lesssim 
\int_{|y|<B} \e^2 \left( \delta(\alpha^*) + |b|+ |p|+|q|\right)
+ \int_{|y|<B} |\e| (p^2+ q^2 + b^2)
\\&\lesssim  \frac 1{100}\int_{|y|<B/2} (\e_y^2+\varepsilon^2) +\frac 1{s^4},
\end{align*}

We now recall from \cite{MMR1} the following coercivity result.

\begin{lemma}[Localized viriel estimate]\label{cl:9}
There exists $B_{0}>100$ and $\mu_{3}>0$ such that
if $B\geq B_{0}$, then
$$
\int_{|y|<B/2}   \left( 3 \varepsilon_y^2 + \varepsilon^2 - 5 Q^4 \varepsilon^2 + 20 y Q' Q^3 \varepsilon^2\right) \geq \mu_{3} \int_{|y|<B/2}  \left(  \varepsilon_y^2 + \varepsilon^2 \right)
-  \frac 1 B \int \e^2 e^{-\frac{|y|}{2}}.
$$
\end{lemma}
 
Thus for $\alpha^*$ small enough:
$$(f^{(i)}_{1,1})^\sim   \leq  - \frac { \mu_{3}}   {2 B} \int_{|y|<B/2}  \left(  \varepsilon_y^2 + \varepsilon^2 \right) + \frac 1 {B^2} \int \e^2 e^{-\frac{|y|}{2}}+\frac 1{s^4}.$$ 
The collection of above estimates yields the bound:
\begin{equation}\label{eq:cf11}
   f^{(i)}_{1,1}  \leq  - \frac  {\mu_4} B    \int\left[\psi'_B\e_{yy}^2+\varphi'_{i,B}(\e_y^2+\e^2)\right] + \frac C{s^4},
    \end{equation}
    for some universal $\mu_4>0$ independent of $B$.
    
    \medskip

  \underline{{\em Term $f^{(i)}_{1,2}$}}:  We decompose  $f_{1,2}$ in a suitable way:
\begin{align*}
f^{(i)}_{1,2} & =  2 \left(\frac {{\lambda}_{s}}{{\lambda}}+{b}\right)  \int {\Lambda} Q  (L \varepsilon) -2\left(\frac {{\lambda}_{s}}{{\lambda}}+{b}\right)\int\varepsilon (1-\varphi_{i,B})\Lambda Q
\\
&+ 2b \left(\frac {{\lambda}_{s}}{{\lambda}}+{b}\right)  \int {\Lambda} ( \chi_b P)  \left(-(\psi_B\e_y)_{y}  + \e\varphi_{i,B} - \psi_B Z\right)\\
& + 2 \left(\frac {{\lambda}_{s}}{{\lambda}}+{b}\right)  \int {\Lambda} Q      \left(-(\psi_B)_y\e_y
-(1-\psi_B)\e_{yy}+(1-\psi_B)Z\right) \\&
+ 2\left(\frac {{\lambda}_{s}}{{\lambda}}+{b}\right)  \int {\Lambda} Q  (Z-5Q^4\e)
\end{align*}
Observe from \fref{ortho1}: $$\int \Lambda Q (L\e)=(\e,L\Lambda Q)=-2(\e,Q)=0.$$  We now use the orthogonality conditions $( \varepsilon, y {\Lambda} Q)=0$ and the definition of $\varphi_{i,B}$ to estimate:
$$
	\left| \int {\Lambda} Q \varepsilon (1-\varphi_{i,B}) \right| = \left| \int {\Lambda} Q \varepsilon \left(1-\varphi_{i,B}+ \frac yB\right)\right|
	\lesssim e^{- \frac B 8 } \mathcal{N}_{i,\rm loc}^{\frac 12}  ,
$$
so that  for $B$ large enough:
\begin{align*}
  \left|  \left(\frac {{\lambda}_{s}}{{\lambda}}+{b}\right)  \int {\Lambda} Q\varepsilon (1-\varphi_{i,B}) \right| 
&\lesssim \left( \mathcal{N}_{i,\rm loc}^{\frac 12}  +\frac 1{s^2} \right) e^{- \frac B 8}  \mathcal{N}_{i,\rm loc}^{\frac 12}  
\\ &
\leq 
 \frac 1 {500}  \frac {\mu_4} {B} \mathcal{N}_{i,\rm loc} 
+\frac C{s^4}.
\end{align*}
For the next term in $f_{1,2}^{(i)}$, we first integrate by parts to remove all derivatives on $\e$.
Then, by   the properties of $\varphi_{i,B}$, $\psi_B$, $P$  and $\chi_b$ \eqref{eq:210}, we obtain
for $\alpha^*$ small,
 \bee
&&\left| 2b \left(\frac {{\lambda}_{s}}{{\lambda}}+{b}\right)  \int {\Lambda} ( \chi_b P)  \left(-(\psi_B\e_y)_{y}  + \e\varphi_{i,B} - \psi_B Z\right)\right|\\
& \lesssim & |b| \left( \mathcal{N}_{i,\rm loc}^{\frac 12}  + \frac 1{s^2}\right)^2   \leq \frac 1 {500}  \frac {\mu_4} {B } \mathcal{N}_{i,\rm loc} (s)
+\frac C{s^4}.
\eee
Next, integrating by parts,   using   the exponential decay of $Q$ and since $\psi_B(y)\equiv 1$ on $[-\frac B2,\infty)$:
\bee
&&\left|  \left(\frac {{\lambda}_{s}}{{\lambda}}+{b}\right)  \int {\Lambda} Q      \left(-(\psi_B)_y\e_y
-(1-\psi_B)\e_{yy}+(1-\psi_B)Z\right) \right|\\
& \lesssim & \left( \mathcal{N}_{i,\rm loc}^{\frac 12}  + \frac 1{s^2}\right)^2  \leq \frac 1 {500}  \frac {\mu_4} {B} \mathcal{N}_{i,\rm loc} +\frac C{s^4},
\eee
and finally
\begin{align*}
&\left| \left(\frac {{\lambda}_{s}}{{\lambda}}+{b}\right)  \int {\Lambda} Q  \left[Z-5Q^4\e\right]
\right|  \lesssim \left( \mathcal{N}_{i,\rm loc}^{\frac 12} + \frac 1{s^2}\right) 
\left( \mathcal{N}_{i,\rm loc} + \frac 1{s^2}\right) 
\\
& \leq \frac 1 {500}  \frac {\mu_4} {B } \mathcal{N}_{i,\rm loc} (s)
+\frac C{s^4}.
\end{align*}
The collection of above estimates yields the bound:
 $$
 |f^{(i)}_{1,2}|\leq \frac 1 {100}  \frac {\mu_4} {B} \mathcal{N}_{i,\rm loc} + \frac C{s^4}.
 $$
\underline{{\em Term $f^{(i)}_{1,3}$}}: Integrating by parts, we claim the identity
\begin{align*}
& -\frac 16 \int \psi_B' \left[(Q_b+p\Yz+q+\e)^6 -Q_b^6 - 6Q_b^5\e - 6Q^5(p\Yz+q)
-6 Q_b q^5\right]\\
& = \int \psi_B (Q_b)_y Z - 5 \int \psi_B Q' Q^4 \e + \int (Q_b^5-Q^5)_y \e\psi_B+  5\int (Q_b-Q)_y Q^4(p\Yz+q) \psi_B \\
&+\int \psi_B p \Yz' Z +\int (Q^5-Q_b^5) p\Yz' \psi_B + 5\int Q^4(p\Yz+q) p\Yz'\psi_B
+\int \psi_B p \Yz' q^5 Q_b\\
& +\int \psi_B q_y \left[(Q_b+p\Yz+q+\e)^5 -Q^5\right]
+\int \psi_B \e_y Z + 5\int \psi_B \e_y Q^4 (p\Yz+q).
\end{align*}
Therefore,
\bee
f^{(i)}_{1,3}& = &   \left(\xsl-1\right) \bigg\{2\int Q'\left[L\e-\psi_B' \e_y + (1-\psi_B)\e_{yy}-\e(1-\varphi_{i,B})\right]\\
& + & 2 \int(Q_b-Q+\e+p\Yz)_y\left[-\psi_B'\e_y-\psi_B\e_{yy}+\e\varphi_{i,B}\right]\\
& + & \frac 13 \int\psi_B'\left[(Q_b+p\Yz+q+\e)^6-Q_b^6-6Q_b^5\e-6Q^5(p\Yz+q)
-6 Q_b q^5\right]\\
& + & 2  \int\e\psi_B(Q_b^5-Q^5)_y+ 10    \int \psi_B (Q_b-Q)_y Q^4 (p\Yz+q)\\
& + & 2   \int \psi_B(Q^5-Q_b^5) p\Yz' +10 \int Q^4 (p\Yz+q) p\Yz' \psi_B
+2 \int \psi_B p\Yz' q^5 Q_b\\
& + & 2    \int \psi_Bq_y \left[ (Q_b+p\Yz+q+\e)^5-Q^5-5Q_b q^4\right] +10\int \psi_B \e_y Q^4(p\Yz+q)\bigg\}.
\eee
The first term is treated using the cancellation $L Q'=0$ and the orthogonality conditions $(\varepsilon, \Lambda Q)=(\e,Q)=0$, so that $(yQ',\e)=0$. Thus, by the definitions of $\varphi_{i,B}$ and $\psi_B$,
\bee
&&\left|2\left(\xsl-1\right) \int Q'\left[L\e-\psi_B' \e_y + (1-\psi_B)\e_{yy}-\e(1-\varphi_{i,B})\right]\right| 
\\
&&=\left|2\left(\xsl-1\right) \int Q'\left[-\psi_B' \e_y + (1-\psi_B)\e_{yy}-\e\left(1+\frac yB-\varphi_{i,B}\right)\right]\right| 
\\&& \lesssim   \left(\mathcal{N}_{i,\rm loc}^{\frac 12}+ \frac 1{s^2}\right)e^{-\frac B {10}}\mathcal{N}_{i,\rm loc}^{\frac 12} \leq   \frac 1 {500}  \frac {\mu_4} {B} \mathcal{N}_{i,\rm loc}+\frac C{s^4}.
\eee
Then, as before, integrating by parts, and using Cauchy-Schwarz inequality,
\bee
&&\left|2b \left(\xsl-1\right)\int (\chi_bP)_y\left[-\psi_B'\e_y-\psi_B\e_{yy}+\e\varphi_{i,B}\right]\right|\\\
&&\lesssim
\frac 1{s} \left(\mathcal N_{i,\rm loc}^{\frac 12} + \frac 1{s^2} \right)
B^{\frac 12} \mathcal N_{i,\rm loc}\leq \frac 1 {500}  \frac {\mu_4} {B} \mathcal{N}_{i,\rm loc}+\frac C{s^4}.
\eee
\bee
&&\left|2\left(\xsl-1\right)\int  \e_y\left[-\psi_B'\e_y-\psi_B\e_{yy}+\e\varphi_{i,B}\right]\right|
\\ &&\lesssim \delta(\alpha^*) \int   \left(\varepsilon_y^2+\e^2\right)   \varphi'_{i,B}\leq \frac1{500} \frac  {\mu_4}{B}\int   \left(\varepsilon_y^2+\e^2\right)   \varphi'_{i,B}+\frac C{s^4}.
\eee
\begin{align*}
& \left|\xsl-1\right| \left| \int p\Yz' \left[ - \psi_B' \e_y
- \psi_B \e_{yy} + \e \varphi_{i,B}\right]\right|\\
& \lesssim \left(\mathcal N_{i,\rm loc}^{\frac 12} + \frac 1{s^2} \right)
\frac 1s \mathcal N_{i,\rm loc}^{\frac 12} 
\leq \frac1{500} \frac  {\mu_4}{B}\int   \left(\varepsilon_y^2+\e^2\right)   \varphi'_{i,B}+\frac C{s^4}.
\end{align*}
 In conclusion for $f^{(i)}_{1,3}$, on gets
$$
|f^{(i)}_{1,3}| \leq \frac1{50} \frac  {\mu_4}{B}\int   \left(\varepsilon_y^2+\e^2\right)   \varphi'_{i,B}+ \frac C{s^4},
$$
for $B$ large enough and $\alpha^*$ small enough, $s_0$ large enough.

\medskip

\underline{{\em Term $f^{(i)}_{1,4}$}}:  Recall
$$
f^{(i)}_{1,4} = - 2{b}_{s}  \int   \left(\chi_b   + \gamma   y (\chi_b)_y\right) P  \left(   - \psi_B\varepsilon_{yy} -\psi_B'\e_y+ \varepsilon \varphi_{i,B}-\psi_BZ   \right).$$
 We estimate after   integrations by parts
 \bee
 \left| \int   \left(\chi_b   + \gamma   y (\chi_b)_y\right) P  \left(   - \psi_B\varepsilon_{y } \right)_y\right|&\lesssim& \int|\e|\left|(\psi_B (  (\chi_b   + \gamma   y (\chi_b)_y) P)_y)_y\right| \\
 & \lesssim & B^{\frac 12}\mathcal{N}_{i,\rm loc}^{\frac 12}.
 \eee
 $$\left|\int   \left(\chi_b   + \gamma   y (\chi_b)_y\right) P \varepsilon \varphi_{i,B}\right|\lesssim B^{\frac 12}\mathcal{N}_{i,\rm loc}^{\frac 12}.$$ The estimate of the nonlinear term follows from   \fref{nonlinearsobolev} and $\psi\leq (\varphi'_i)^2$ for $y<-\frac 12$:
 \bee
 && \left|\int   \left(\chi_b   + \gamma   y (\chi_b)_y\right) P \psi_BZ\right|  \lesssim  \int\psi_B((|Q_b|^4|+p^4+q^4) \e|+|\e|^5
 +b^2 + p^2 + q^2)\\
 && \lesssim    B^{\frac 12}\left( \int (|\e|^2+ |\e|^6) \psi_B\right)^{\frac 12} +\frac 1{s^2}  \lesssim B^{\frac 12}
\left(\int   \left(\varepsilon_y^2+\e^2\right)   \varphi'_{i,B} \right)^{\frac 12} 
+\frac 1{s^2}  .
 \eee 
Together with \eqref{BSparam}, these estimates yield the bound:
$$|f_{1,4}| 
\leq  \frac 1 {500}  \frac {\mu_4} {B } 
 \int   \left(\varepsilon_y^2+\e^2\right)   \varphi'_{i,B}   +\frac C{s^4}.$$
 
 \medskip
 
\underline{{\em Term $f^{(i)}_{1,5}$}}:  Recall:
$$
f^{(i)}_{1,5}=    
 2  \int  \Psi_b    \left(   -(\psi_B\varepsilon_{y})_y  + \varepsilon \varphi_{i,B}-\psi_B Z   \right).
$$
We now rely on \eqref{BSbsps} to estimate by integration by parts and Cauchy-Schwarz's inequality,
$$
\left|  \int  (\Psi_b)_y    \psi_B\varepsilon_{y}\right| \lesssim    B^{\frac 12}  b^2 \mathcal{N}_{i,\rm loc}^{\frac12} \leq  \frac 1 {500}  \frac {\mu_4} {B} \mathcal{N}_{i,\rm loc} +\frac C{s^4}.
$$
By \eqref{eq:202}  and   the exponential decay of $\varphi_{i,B}$ in the left,
$$\left|  \int  \Psi_b \varphi_{i,B}\e\right|\lesssim \left( b^2 B^{\frac 12}+ e^{-\frac 1{2 |b|^{\gamma}}}\right)|b|^{1+\gamma}\mathcal{N}_{i,\rm loc}^{\frac12}\leq  \frac 1 {500}  \frac {\mu_4} {B } \mathcal{N}_{i,\rm loc} +C|b|^{4}.$$ For the nonlinear term, similarly and using \eqref{nonlinearsobolev},
\bee
 \left| \int  \Psi_b   \psi_B Z   \right|  \leq  \frac 1 {500}  \frac {\mu_4} {B} \int   \left(\varepsilon_y^2+\e^2\right)   \varphi'_{i,B}  +\frac C{s^4}.
\eee
The collection of above estimates yields the bound:
$$
| f^{(i)}_{1,5}| \leq \frac 1 {100}  \frac {\mu_4} {B} \int   \left(\varepsilon_y^2+\e^2\right)   \varphi'_{i,B}+\frac C{s^4}.$$
\medskip
 
\underline{{\em Term $f^{(i)}_{1,6}$}}: Recall that this term writes
\begin{align*}
\left| \int \left( -p_s \Yz + (5Q^4 (p-q))_y + \lsl p\Lambda \Yz\right)
\left( - (\psi_B \e_y)_y + \e \varphi_{i,B} - \psi_B Z\right) \right|
\end{align*} 
By \eqref{BSbsps}, \eqref{cl.2}, \eqref{BSparam},
\begin{align*}
& \left|  \left( -p_s \Yz + (5Q^4 (p-q))_y + \lsl p\Lambda \Yz\right)  \right|
+\left|  \left( -p_s \Yz + (5Q^4 (p-q))_y + \lsl p\Lambda \Yz\right)_y  \right|
\\ & \lesssim e^{-\frac {|y|}{2}} \left( \frac 1{s^2} + \int \e^2 e^{-\frac {|y|}{2}} \right), 
\end{align*}
and thus
$$
| f^{(i)}_{1,6}| \lesssim \delta(\alpha^*) \int \e^2 e^{-\frac {|y|}{2} } 
+\frac 1{s^4}.
$$

\medskip
{\bf step 4} $f_2^{(i,j)}$ term. 
 
Recall:
\bee
f^{(i,j)}_2 &= & 2 {\frac{{\lambda}_s}{{\lambda}}}  \int    {\Lambda} \varepsilon  \left( - (\psi_B\varepsilon_y)_y +  \varepsilon\varphi_{i,B} -\psi_BZ \right)+\frac js\mathcal F_{i}.
\eee
We integrate by parts to compute:
\begin{align*}
 & \int {\Lambda} \varepsilon (\psi_B\varepsilon_{y})_{y} =  -\int \e_y^2 \psi_B + \frac 12 \int \e^2_y y \psi_B',\\
 & \int  ({\Lambda} \varepsilon )    \varepsilon \varphi_{i,B} =- \frac 12 \int     \varepsilon^2    y\varphi_{i,B}',\\
 & \int {\Lambda} \varepsilon\psi_BZ
 = 
\int \left( \frac \e 2 + y\e_y\right) \psi_B  \left[ (Q_b+p\Yz+q+\e)^5 - Q_b^5 
-5 Q^4 (p\Yz+q)-q^5\right]\\
& = \int \frac \e 2 \psi_B \left[ (Q_b+p\Yz+q+\e)^5 - Q_b^5 
-5 Q^4 (p\Yz+q)-q^5\right]\\
& - \int (y\psi_B)_y  \bigg[ \frac {(Q_b+p\Yz+q+\e)^6} 6  
- \frac {(Q_b+p\Yz+q)^6} 6  - Q_b^5 \e \\&\qquad  - 5 Q^4(p\Yz+q) \e- Q_b^5 (p\Yz+q) -q^5 \e\bigg]\\
& - \int y \psi_B (Q_b)_y \big[ (Q_b+p\Yz+q+\e)^5 - (Q_b+p\Yz+q)^5 - 5Q_b^4 \e - 5 Q_b^4 (p\Yz+q) \e \\ &\qquad - 20 Q_b^3 (p\Yz+q) \e\big] \\
& - \int y \psi_B p\Yz' \left[ (Q_b+p\Yz+q+\e)^5 - (Q_b+p\Yz+q)^5 - 5Q^4 \e\right]\\
&+5 \int (Q_b^4 - Q^4)_y (p\Yz+q) \e y \psi_B\\
& - \int y \psi_B \left[ (Q_b+p\Yz+q+\e)^5 - (Q_b+p\Yz + q)^5 - 5Q^4 \e - Q_b^5 - 5 q^4 \e\right].
\end{align*}
Thus,
 \begin{align*}
& f^{(i,j)}_2 = 2 \lsl \bigg\{   \int \e_y^2 \psi_B -\frac 12 \int \e^2_y y \psi_B' - \frac 12 \int     \varepsilon^2    y\varphi_{i,B}'
 \\& + \int \frac \e 2 \psi_B \left[ (Q_b+p\Yz+q+\e)^5 - Q_b^5 
-5 Q^4 (p\Yz+q)-q^5\right]\\
& - \int (y\psi_B)_y  \bigg[ \frac {(Q_b+p\Yz+q+\e)^6} 6  
- \frac {(Q_b+p\Yz+q)^6} 6  - Q_b^5 \e \\ &  \qquad  - 5 Q^4(p\Yz+q) \e- Q_b^5 (p\Yz+q) -q^5 \e\bigg]\\
& - \int y \psi_B (Q_b)_y \big[ (Q_b+p\Yz+q+\e)^5 - (Q_b+p\Yz+q)^5 - 5Q_b^4 \e - 5 Q_b^4 (p\Yz+q) \e \\ &\qquad - 20 Q_b^3 (p\Yz+q) \e\big] \\
& - \int y \psi_B p\Yz' \left[ (Q_b+p\Yz+q+\e)^5 - (Q_b+p\Yz+q)^5 - 5Q^4 \e\right]\\
&+5 \int (Q_b^4 - Q^4)_y (p\Yz+q) \e y \psi_B\\
& - \int y \psi_B \left[ (Q_b+p\Yz+q+\e)^5 - (Q_b+p\Yz + q)^5 - 5Q^4 \e - Q_b^5 - 5 q^4 \e\right]\bigg\} +\frac js\mathcal F_{i}.
  \end{align*}
  For $y<-B$, we use the exponential decay of $\psi_B, \varphi_{i,B}$ and   \fref{defphi4}
to estimate:
\bee
&&\int_{y<-\frac B2}(\psi_B+|y|\psi'_B+\varphi_{i,B})(\e_y^2+\e^2) +|y|\varphi'_{i,B} \e^2\\
&& \lesssim  \int_{y<-\frac B2} \e_y^2 \varphi'_{i,B}+\int_{y<-\frac B2}|y|\varphi'_{i,B}\e^2\\
&&\lesssim  \int   \varepsilon_y^2    \varphi'_{i,B}  + \left(\int_{y<-\frac B2}|y|^{100}e^{\frac yB}\e^2\right)^{\frac{1}{100}}\left(\int_{y<-\frac B2}e^{\frac yB}\e^2\right)^{\frac{99}{100}} \\
&& \lesssim \int   \varepsilon_y^2    \varphi'_{i,B} +\mathcal{N}_{i,\rm loc}^{\frac{9}{10}},
\eee
where we have used $\int_{y<-\frac B2}|y|^{100}e^{\frac yB}\e^2\leq \|\e\|_{L^2}^2\leq \delta(\alpha^*)$.

 Together with similar estimates for the other terms,
this yields the bound:
$$|(f_2^{(i,j)})^<|\lesssim (\frac 1s+\mathcal{N}_{i,\rm loc}^{\frac 12})\left(\int   \varepsilon_y^2    \varphi'_{i,B} +\mathcal{N}_{i,\rm loc}^{\frac{9}{10}}+\frac 1{s^4}\right)\lesssim \delta(\alpha^*) \int  ( \varepsilon_y^2 +\e^2)   \varphi'_{i,B}  +
\frac 1{s^4}.$$

The middle term $f_{2}^{(i,j)}$ is also   estimated as follows
$$
|f_{2}^{(i,j)}|\leq \delta(\alpha^*) \int  ( \varepsilon_y^2 +\e^2)   \varphi'_{i,B}  +
\frac 1{s^4}.
$$

It only remains to estimate $(f_{2}^{(i,j)})^>$.
Most terms in $(f_{2}^{(i,j)}$ are easily estimated similarly as before. We focus on the following two delicate terms (because of weight at $+\infty$):
$$\int_{y> \frac B2}   \e^2    y\varphi_{i,B}', \quad
\frac j{s} \int_{y>\frac B2} \e^2 \varphi_{i,B}.$$
The function $\psi_B$ being bounded, the other terms are easier.

First, using \eqref{weightedonebis},
\begin{align*}
 \frac 1 s\int_{y>\frac B2} \e^2 \varphi_{i,B} & \lesssim
\frac 1s \left( 1+ \frac 1{\l^{\frac {10}9}}\right) \mathcal N_{2,\rm loc}^{\frac 89}\\
& \lesssim s^{\frac {10}9 \beta -1}\mathcal N_{2,\rm loc}^{\frac 89} 
\leq \frac 1 {100} \frac {\mu_4}{B} \mathcal N_{2, \rm loc} + s^{10 \beta -9}.
\end{align*}
Second, using $\mathcal N_{2,\rm loc}  \lesssim s^{-\frac 52 }$ and
$\beta<\frac 78$,
\begin{align*}
\int_{y> \frac B2}   \e^2    y\varphi_{i,B}' & \lesssim  \left(\frac 1s + \left( \int \e^2 e^{-\frac{|y|}{10}}\right)^{\frac 12} \right) \int_{y>\frac B2}   \e^2 \varphi_{i,B}\\
& \lesssim  \left(\frac 1s + \mathcal N_{2,\rm loc}^{\frac 12}  \right) s^{\frac {10}9 \beta }\mathcal N_{2,\rm loc}^{\frac 89} \leq \frac 1 {100} \frac {\mu_4}{B} \mathcal N_{2, \rm loc} + s^{10 \beta -9}.
\end{align*}

{\bf step 5} $f_3^{(i)}$ term.

\begin{align*}
f_3^{(i)}&= 2  \int \psi_B(Q_b)_{s} \left[(\varepsilon+Q_b+p\Yz+q)^5 - (Q_b+p\Yz+q)^5 - 5 \varepsilon Q_b^4\right]\\
& - 2 \int \psi_B (p_s \Yz+ q_s) \left[(\varepsilon+Q_b+p\Yz+q)^5  -(Q_b+p\Yz+q)^5 - 5 \varepsilon Q^4\right]\\
&+ 10 \int \psi_B q_s q^4 \e.
\end{align*}

First,
$$
|(Q_b)_s|=\left|b_s P \left( \chi(|b|^{\gamma} y) + \gamma |b|^\gamma y \chi'(|b|^{\gamma} y)\right)\right|\lesssim |b_s| .
$$

\begin{align*}
&\left| \int \psi_B(Q_b)_{s} \left[(\varepsilon+Q_b+p\Yz+q)^5 - (Q_b+p\Yz+q)^5 - 5 \varepsilon Q_b^4\right]\right| \\
  & \lesssim  |b_s|\int\psi_B \left(\e^2(|Q_b|^3+|p|^3+|q|^3+\delta(\alpha^*))  +|\e| (e^{-\frac {|y|}2} p^4+q^4) \right)\\
& \lesssim    \left(\frac 1{s^2} +  \mathcal{N}_{i,\rm loc}\right)  \left( \int (\e_y^2+\varepsilon^2)  \psi_{B} + \left( \int \e^2 \psi_B \right)^{\frac 12 } \left( \int q^8\psi_B \right)^{\frac 12}\right)\\
& \lesssim  \left(\delta(\alpha^*) + \delta(s_0^{-1}) \right) \int \varepsilon^2   \varphi_{i,B}' +\frac C{s^4}.
\end{align*}
For the next two terms, we first remark that by explicit computations:
$$
\int \psi_B (p_s \Yz+ q_s)^2\lesssim \frac 1{s^2}.
$$
Thus, as before,
\begin{align*}
& \left|\int \psi_B (p_s \Yz+ q_s) \left[(\varepsilon+Q_b+p\Yz+q)^5  -(Q_b+p\Yz+q)^5 - 5 \varepsilon Q^4\right] \right| \\
& \lesssim    \left(\delta(\alpha^*) + \delta(s_0^{-1}) \right) \int \varepsilon^2   \varphi_{i,B}' +\frac C{s^4}.
\end{align*}
Finally,
$$
\left|  \int \psi_B q_s q^4 \e \right| \lesssim \left( \int \psi_B (q_s q^4)^2 \right)^{\frac 12}
\left( \int \psi_B  \e^2 \right)^{\frac 12}\lesssim \int \psi_B  \e^2 +\frac 1{s^4}.
$$
\medskip

{\bf step 6} Proof of \fref{lowerbound}.

We proceed as in \cite{MMR1}. 
Recall that for $B$ large enough, $\mu>0$,
$$\int \psi_B\varepsilon_y^2 + \varphi_{i,B}\varepsilon^2- 5 \psi_B Q^4\e^2 \geq \mu \mathcal N_i.$$ 
We only have to estimate the error term as follows. 
For $s_0$ large enough, and $\alpha^*$ small enough,
\begin{align*}
& \left| \int \left( (Q_b{+}\e{+}p\Yz{+}q)^6 {-} (Q_b{+}p\Yz{+}q)^6  
{-} 6 \e (Q_b^5 {+} q^5 {+} 5 Q^4(p\Yz{+}q)) {-} 6 Q^4 \e^2\right) \psi_B\right| \\
& \lesssim
\frac 1 {s^2}\left( \int \psi_B  \e^2 \right)^{\frac 12}+
\frac 1s \int \psi_B  \e^2 \leq \frac \mu {10} \mathcal N_i+ \frac 1 {s^4} .
\end{align*}
This concludes the proof of Lemma \ref{propasymtp}.


\begin{thebibliography}{10}
\bibitem{Aubin} T. Aubin, \'Equations diff\'erentielles non lin\'eaires et probl\`eme de Yamabe concernant la courbure scalaire, J. Math. Pure Appl. \textbf{55:3} (1976), 269--296.
\bibitem{BW} J. Bourgain and  W. Wang,   Construction of blowup solutions for the nonlinear Schr\"odinger equation with critical nonlinearity. Ann. Scuola Norm. Sup. 
Pisa Cl. Sci. (4) 25 (1997),   197--215 (1998). 
\bibitem{Cotekdv} R. C\^ote,  Construction of solutions to the $L^2$-critical KdV equation with a given asymptotic behaviour, Duke Math. J. 138 (2007),   487--531.
\bibitem{MMC} R. C\^ote, Y. Martel and F. Merle,   Construction of multi-soliton solutions for the $L^2$-supercritical gKdV and NLS equations, arXiv:0905.0470.
\bibitem{CoteZaag} R. C\^ote and H. Zaag,
Construction of a multi-soliton blow-up solution to the semilinear wave equation in one space dimension. Preprint.
\bibitem{DK} Donninger, R.; Joachim, K., Nonscattering solutions and blowup at infinity for the critical wave equation, preprint, arXiv:1201.3258
\bibitem{DKM} Duyckaerts, T., Kenig, C., Merle, F. Universality of blow-up profile for small radial type II
blow-up solutions of energy-critical wave equation, preprint, arXiv:0910.2594.
\bibitem{DKM2} Duyckaerts, T., Kenig, C., Merle, F. Universality of the blow-up profile for small type II blowup
solutions of energy-critical wave equation: the non-radial case, preprint, arXiv:1003.0625.
\bibitem{DM2} T. Duyckaerts, F. Merle, Dynamics of threshold solutions for energy-critical wave equation. Int. Math. Res. Pap. IMRP 2007, Art. ID rpn002, 67 pp. (2008).
\bibitem{DM}  T. Duyckaerts and F. Merle,   Dynamic of threshold solutions for energy-critical NLS, Geom. Funct. Anal. 18 (2009),   1787--1840.
\bibitem{DR} T. Duyckaerts and S. Roudenko,  Threshold solutions for the focusing 3D cubic Schr\"odinger equation, Rev. Mat. Iberoam. 26 (2010),   1--56.
\bibitem{FMR} G. Fibich, F.  Merle and  P. Rapha\"el,
Proof of a spectral property related to the singularity formation for the L2 critical nonlinear Schr\"odinger equation. Phys. D \textbf{220} (2006), 1--13.
\bibitem{FKWY} Fila, Marek; King, John R.; Winkler, Michael; Yanagida, Eiji Very slow grow-up of solutions of a semi-linear parabolic equation. Proc. Edinb. Math. Soc. \textbf{54} (2011),  381--400.
\bibitem{GNN} B. Gidas, W.M. Ni and L. Nirenberg, 
Symmetry and related properties via the maximum principle, Comm. Math. Phys. \textbf{68} (1979), 209--243.
\bibitem{GV} J. Ginibre and  G. Velo, On a class of nonlinear Schr\"odinger equations. I. The Cauchy problem, general case. J. Funct. Anal. \textbf{32} (1979),  1--32.
\bibitem{NT2} S. Gustafson,   K. Nakanishi  and    T.-P. Tsai,   Asymptotic stability, concentration and oscillations in harmonic map heat flow, Landau Lifschitz and Schr\"odinger maps on $\RR^2$, Comm. Math. Phys.  \textbf{300}  (2010), 205-242.
\bibitem{RH} M. Hillairet and   P. Rapha\"el, Smooth type II blow up solutions to the four dimensional energy critical wave equation, arXiv:1010.1768.
\bibitem{Kato} T. Kato, On the Cauchy problem for the (generalized) Korteweg-de Vries equation.  Studies in applied mathematics,  93--128, Adv. Math. Suppl. Stud., 8, Academic Press, New York, 1983.
\bibitem{KM} C.E.
Kenig, F. Merle,
Global well-posedness, scattering and blow-up for the energy-critical, focusing, nonlinear Schr\"odinger equation in the radial case.  
Invent. Math. \textbf{166} (2006) 645--675.
\bibitem{KPV} C.E. Kenig, G. Ponce and L. Vega, Well-posedness and scattering results for the generalized Korteweg--de Vries equation via the contraction principle, Comm. Pure Appl. Math. \textbf{46}, (1993) 527--620. 
\bibitem{KPV2} C.E. Kenig, G. Ponce and L. Vega, On the concentration of blow up solutions for the generalized KdV equation critical in $L^2$. Nonlinear wave equations (Providence, RI, 1998), 131--156, Contemp. Math., \textbf{263}, Amer. Math. Soc., Providence, RI, 2000.
\bibitem{KM1} C.E. Kenig  and  F. Merle,  Global well-posedness, scattering and blow-up for the energy-critical, focusing, non-linear Schr\"odinger equation in the radial case. Invent. Math. \textbf{166} (2006),   645--675.
\bibitem{KKSV} R. Killip,  S. Kwon, S. Shao, M. Visan,  On the mass-critical generalized KdV equation. Discrete Contin. Dyn. Syst. \textbf{32} (2012),   191-221.
\bibitem{KMR}
J. Krieger, Y.  Martel and P. Rapha\"el,  Two-soliton solutions to the three-dimensional gravitational Hartree equation. Comm. Pure Appl. Math. \textbf{62} (2009),   1501--1550.
 \bibitem{KNS} J. Krieger, K. Nakanishi  and  W. Schlag,  Global dynamics away from the ground state for the energy-critical nonlinear wave equation,  arXiv:1010.3799.
\bibitem{KSNLS} J. Krieger and W. Schlag, Non-generic blow-up solutions for the critical focusing NLS in 1-D, J. 
Eur. Math. Soc. (JEMS) \textbf{11} (2009),  1--125.
\bibitem{KST} J. Krieger, W.  Schlag and  D. Tataru,  Renormalization and blow up for charge one equivariant critical wave maps. Invent. Math. \textbf{171} (2008),   543--615.
\bibitem{KST2} J. Krieger, W.  Schlag and  D. Tataru,   Slow blow-up solutions for the $H^1(\RR^3)$ critical focusing semilinear wave equation, Duke Math. J. \textbf{147} (2009), 1--53.
\bibitem{Kw} M. K. Kwong, Uniqueness of positive solutions of $\Delta u - u + u^p=0$ in $\RR^n$, Arch. Rational Mech. Anal. \textbf{105} (1989), 243--266.
\bibitem{Ma1}  Y. Martel, Asymptotic $N$--soliton--like solutions of the subcritical and critical generalized Korteweg--de Vries equations, Amer. J. Math.  \textbf{127} (2005), 1103-1140.
\bibitem{MMjmpa} Y. Martel and F. Merle, 
A Liouville theorem for the critical generalized Korteweg--de Vries equation, 
J. Math. Pures Appl. \textbf{79} (2000), 339--425.
\bibitem{MMgafa} Y. Martel and F. Merle, Instability of solitons for the critical generalized Korteweg-de Vries equation.  Geom. Funct. Anal.  \textbf{11}  (2001),   74--123.
\bibitem{MMannals} Y. Martel and F. Merle,
Stability of blow up profile and lower bounds for blow up rate for the critical generalized KdV equation,
Ann. of Math. \textbf{155} (2002), 235--280.
\bibitem{MMjams}
Y. Martel and F. Merle, Blow up in finite time and dynamics of blow up solutions for the $L\sp 2$-critical generalized KdV equation,  J. Amer. Math. Soc.  \textbf{15}  (2002),   617--664.
\bibitem{MMduke} Y. Martel and F. Merle,   Nonexistence of blow-up solution with minimal $L^2$-mass for the critical gKdV equation. Duke Math. J. \textbf{115} (2002),  385--408.
\bibitem{MMnls}
Y. Martel and F.  Merle, Multi solitary waves for nonlinear Schr\"odinger equations. Ann. Inst. H. Poincar\'e Anal. Non Lin\'eaire \textbf{23} (2006),   849--864.
\bibitem{MMcol1} Y. Martel and F. Merle, Description of the collision of two solitons for the quartic gKdV equation, Annals of Math. \textbf{174} (2011), 757--857.
\bibitem{MMR1} Y. Martel, F. Merle and P. Rapha\"el, 
Blow up for the critical gKdV equation I: dynamics near the soliton. Preprint.
\bibitem{MMR2} Y. Martel, F. Merle and P. Rapha\"el, 
Blow up for the critical gKdV equation II: minimal mass solution. Preprint.
\bibitem{MMT} Y. Martel, F. Merle, T.-P. Tsai, Stability in $H^1$ of the sum of $K$ solitary waves for some nonlinear Schr\"odinger equations, Duke Math. J. \textbf{133} (2006),    405--466.
\bibitem{Me0}
F. Merle,  Construction of solutions with exactly $k$ blow-up points for the Schr\"odinger equation with critical nonlinearity, Comm. Math. Phys. \textbf{129} (1990),   223--240.
\bibitem{Mduke} F. Merle, Determination of blow-up solutions with minimal mass for nonlinear Schr\"odinger equations with critical power, Duke Math. J. \textbf{69} (1993),  427--454.
\bibitem{Me} F. Merle, Nonexistence of minimal blow-up solutions of equations $iu_t = -\Delta u -k(x)|u|^{\frac 4N}u$ in $\RR^N$, Ann. Inst. H. Poincar\'e Phys. Th\'eor. \textbf{64} (1996),   33--85.
\bibitem{Mjams} F. Merle, Existence of blow-up solutions in the energy space for the critical generalized KdV equation.  J. Amer. Math. Soc.  \textbf{14}  (2001),  555--578.
\bibitem{MRgafa} F. Merle and P. Rapha\"el, Sharp upper bound on the blow up rate for the critical nonlinear Schr\"odinger equation,   Geom. Func. Anal. \textbf{13} (2003), 591--642.
\bibitem{MRinvent} F. Merle and P. Rapha\"el, On universality of blow-up profile for $L\sp 2$ critical nonlinear Schr\"odinger equation. Invent. Math. \textbf{156} (2004), 565--672.
\bibitem{MRannals} F. Merle and P. Rapha\"el, The blow up dynamics and  upper bound on the blow up rate for the critical nonlinear Schr\"odinger equation, Ann. of Math. \textbf{161} (2005), 157--222.
\bibitem{MRcmp} F. Merle and P. Rapha\"el, Profiles and quantization of the blow up mass for critical nonlinear Schr\"odinger equation, Commun. Math. Phys. \textbf{253} (2005), 675--704.
\bibitem{MRjams} F. Merle and P. Rapha\"el, On a sharp lower bound on the blow-up rate for the $L\sp 2$ critical nonlinear Schr\"odinger equation.  J. Amer. Math. Soc.  \textbf{19}  (2006),   37--90.
\bibitem{MRS} F. Merle, P. Rapha\"el and J. Szeftel,
The instability of Bourgain-Wang solutions for the $L^2$ critical NLS, preprint arXiv:1010.5168
\bibitem{MRR}  F. Merle, P. Rapha\"el and  I. Rodnianski,
    Blow up dynamics for smooth data equivariant solutions to the energy critical Schrodinger map problem.
preprint   arXiv:1102.4308  
\bibitem{NS1} K. Nakanishi and  W. Schlag,  Global dynamics above the ground state energy for the focusing nonlinear Klein-Gordon equation,
J. Differential Equations \textbf{250} (2011),   2299--2333.
\bibitem{NS2} K. Nakanishi and  W. Schlag,   Global dynamics above the ground state energy for the cubic NLS equation in 3D,  arXiv:1007.4025.
\bibitem{Rannalen} P. Rapha\"el,  Stability of the log-log bound for blow up solutions to the critical non linear Schr\"odinger equation. Math. Ann. \textbf{331} (2005),  577--609. 
\bibitem{Rzurich}     P. Rapha\"el, Stability and blow up for the nonlinear Schrodinger equation, Lecture notes for the Clay summer school on evolution equations, ETH, Zurich (2008)
\bibitem{RR2009}  P. Rapha\"el and I. Rodnianski, Stable blow up dynamics for the critical co-rotational Wave Maps and equivariant Yang-Mills problems. To appear in Publications scientifiques de l'IHES.  arXiv:0911.0692
\bibitem{Rstud} P. Rapha\"el and R. Schweyer, Stable blow up dynamics for the 1-corotational harmonic heat flow, preprint 2011.
\bibitem{RS2010} P. Rapha\"el and J. Szeftel, Existence and uniqueness of minimal blow up solutions to an inhomogeneous mass critical NLS. To appear in J. Amer. Math. Soc. \textbf{23} (2011). Preprint arXiv:1001.1627
\bibitem{PSS}
 M. J. Landman, G. C. Papanicolaou,  C.  Sulem and  P.-L. Sulem,  Rate of blowup for solutions of the nonlinear Schr\"odinger equation at critical dimension. Phys. Rev. A (3) \textbf{38} (1988),   3837--3843.
\bibitem{PE} G. Perelman, On the formation of singularities in solutions of the critical nonlinear Schr\"odinger equation, Ann. Henri Poincar\'e \textbf{2} (2001), 605--673.
  \bibitem{RodSter} I. Rodnianski, J. Sterbenz,  On the formation of singularities in the critical $O(3)$ $\sigma$-model, Ann. of Math. (2) \textbf{172} (2010),   187--242.
  \bibitem{S} Schweyer R., Type II blow bp for the four dimensional energy critical semi linear heat equation,  arXiv:1201.2907. 
  \bibitem{Shao} S. Shao,  The linear profile decomposition for the Airy equation and the existence of maximizers for the Airy Strichartz inequality. Anal. PDE \textbf{2} (2009),   83--117.
  \bibitem{Talenti} G. Talenti, Best constant in Sobolev inequality, Ann. Mat. Pura Appl. \textbf{110} (1976), 535--372.
\bibitem{W1983} M.I. Weinstein, Nonlinear Schr\"odinger equations and sharp interpolation estimates, Comm. Math. Phys. \textbf{87} (1983), 567--576.
\bibitem{W1985}  M.I.  Weinstein, Modulational stability of ground states of nonlinear Schr\"odinger equations, SIAM J. Math. Anal., \textbf{16} (1985), 472--491.
\bibitem{W1986} M.I. Weinstein, Lyapunov stability of ground states of nonlinear dispersive evolution equations, Comm. Pure Appl. Math. \textbf{39}, (1986) 51--68.
\end{thebibliography}
\end{document}